\documentclass{amsart}
\usepackage[utf8]{inputenc}

\usepackage[T1]{fontenc}
\usepackage[english]{babel}
\usepackage{lmodern}
\usepackage{amsfonts}
\usepackage{amsmath}
\usepackage{graphicx}
\usepackage{changepage}
\usepackage{amsthm}
\usepackage{tikz}
\usepackage{enumitem}
\usepackage{amssymb}

\usepackage{upref}
\usepackage[backref,bookmarks=false]{hyperref}
\usepackage{mathrsfs}
\usepackage{url} 
\usepackage{imakeidx}
\usepackage{bbm}

\usepackage{setspace}

\usepackage{color}

\newtheorem{theorem}{Theorem}[section]
\newtheorem{defi}[theorem]{Definition}
\newtheorem{notation}[theorem]{Notation}
\newtheorem{prop}[theorem]{Proposition}
\newtheorem{lemma}[theorem]{Lemma}
\newtheorem{cor}[theorem]{Corollary}

\newtheorem{rem}[theorem]{Remark}
\newtheorem{fact}[theorem]{Fact}
\newtheorem*{theoremA}{Theorem A}
\newtheorem*{theoremB}{Theorem B}
\newcommand{\N}{\mathbb{N}}
\newcommand{\diam}{{\rm{diam}}}
\newcommand{\LL}{{\mathcal L}}
\newcommand{\re}{{\rm{Re}}}
\newcommand{\norm}[1]{\left\lVert#1\right\rVert}

\newcommand{\BV}{\mathrm{BV}}
\newcommand{\ben}{\begin{enumerate}}
\newcommand{\een}{\end{enumerate}}
\newcommand{\bitm}{\begin{itemize}}
\newcommand{\eitm}{\end{itemize}}

\newcommand{\bea}{\begin{eqnarray}}
\newcommand{\ba}{\begin{array}}
\newcommand{\bean}{\begin{eqnarray*}}
\newcommand{\ea}{\end{array}}
\newcommand{\eea}{\end{eqnarray}}
\newcommand{\eean}{\end{eqnarray*}}
\newcommand{\beq}{\begin{equation}}
\newcommand{\eeq}{\end{equation}}
\newcommand{\bthm}{\begin{thm}}
\newcommand{\ethm}{\end{thm}}
\newcommand{\blem}{\begin{lem}}
\newcommand{\elem}{\end{lem}}
\newcommand{\bprop}{\begin{prop}}
\newcommand{\eprop}{\end{prop}}
\newcommand{\bcor}{\begin{cor}}
\newcommand{\ecor}{\end{cor}}
\newcommand{\bdfn}{\begin{dfn}}
\newcommand{\edfn}{\end{dfn}}
\newcommand{\brem}{\begin{rem}}
\newcommand{\erem}{\end{rem}}
\newcommand{\bpf}{\begin{proof}}
\newcommand{\epf}{\end{proof}}
\newcommand{\bfact}{\begin{fact}}
\newcommand{\efact}{\end{fact}}
\newcommand{\bobs}{\begin{observation}}
\newcommand{\eobs}{\end{observation}}
\newcommand{\bobsdef}{\begin{observationdefinition}}
\newcommand{\eobsdef}{\end{observationdefinition}}
\newcommand{\bhyp}{\begin{hypothesis}}
\newcommand{\ehyp}{\end{hypothesis}}
\newcommand{\bexa}{\begin{example}}
\newcommand{\eexa}{\end{example}}
\newcommand{\bnot}{\begin{notation}}
\newcommand{\enot}{\end{notation}}

\alph{enumii} \roman{enumiii}

\newcommand\bal{\begin{aligned}}
\newcommand\eal{\end{aligned}}

\begin{document}                                                                                               

\title[Asymptotics of the Hausdorff measure]{Asymptotics of the Hausdorff measure for the Gauss map and its linearized analogue}
\author{Rafał Tryniecki}
 \address{Rafał Tryniecki:University of Warsaw,Institute of Mathematics, ul. Banacha 2, 02-097 Warszawa, Poland}
 \email{rafal.tryniecki@mimuw.edu.pl}
\author{Mariusz Urba\'nski}
\address{Mariusz Urbański:Department of Mathematics
University of North Texas
Denton, TX 76203-1430
USA}
\email{urbanski@unt.edu}
\author{Anna Zdunik}
 \address{Anna Zdunik:University of Warsaw, Institute of Mathematics, ul. Banacha 2, 02-097 Warszawa, Poland}
\email{A.Zdunik@mimuw.edu.pl}

\begin{abstract}
Let $G(x):=\{1/x\}$ be the Gauss map. By $g_n(x)=\frac{1}{x+n}$ we denote its continuous/real analytic inverse branches. We define iterated function system (IFS) $G_n$ by limiting the collection of functions $g_k$, $k\in\N$, to the first $n$ elements, meaning that $G_n = \{g_k \}_{k=1}^n$. We are interested in the asymptotics of the Hausdorff measure of the limit set $J_n$ i. e. set consisting of irrational elements of $[0,1]$ having continued fraction expansion with entries at most $n$. In the first part of the paper, we deal with the piecewise-linear analogue of the Gauss map and resulting IFSs. We prove that 
\[
\lim \limits_{n \to \infty } \frac{1-H_n(J_n)}{1-h_n} \cdot \frac{1}{\ln n} = 1,
\]
where $J_n$ is the limit set of the piecewise-linear analogue of $G_n$, $h_n$ is its Hausdorff dimension and $H_n$ is the value of   $h_n$-dimensional Hausdorff measure of the set $J_n$, $H_n:=H_{h_n}(J_n)$.

In the second part, we focus on the IFS generated by the first $n$ branches of Gauss map and prove, as our main result, that
$$
\lim_{n\to\infty} \frac{1-H_n}{(1-h_n)\ln n}= 1
$$
and equivalently, due to Hensley's result,
$$
\lim_{n\to\infty} \frac{n(1-H_n)}{\ln n}= \frac{6}{\pi^2},
$$
where $J_n$ is the limit set of the system $G_n$, i.e. the set consisting of irrational numbers in $[0,1]$ that continued fraction expansion with entries not exceeding $n$. Similarly as for the piecewise linear map, $h_n$ is the Hausdorff dimension of $J_n$ and $H_n$ is the value of   $h_n$-dimensional Hausdorff measure of the set $J_n$, $H_n:=H_{h_n}(J_n)$.
\end{abstract}

\maketitle

\tableofcontents 
\thanks {This research was funded in whole or  in part by the National Science Centre, Poland, grant no. 2023/49/B/ST1/03015 (A.Zdunik and R.Tryniecki).
The research of the second named author was supported in part by the Simons Grant MPS-TSM-00007114.}


\section{Introduction}

Let 
$$
G(x):=\left\{\frac 1 x \right\}
$$
be the well-known Gauss map. Recall that the map $G$ is closely related to the continued fraction expansion of a point $x\in[0,1]\setminus\mathbb Q$. Namely, if 
$$
x=[x_1,x_2,\dots]
$$ 
is the continued fraction expansion of a point $x$, then the expansion of $G(x)$ is given by applying the left shift to the fractional expansion for $x$,  i.e.  
$$
G(x)=[x_2, x_3,\dots].
$$
The map $G$ is piecewise monotone decreasing and maps every maximal interval of monotonicity $\left(\frac{1}{n+1},\frac 1 n\right]$, $n\in\N$, onto $[0,1)$. The continuous/real analytic inverse branches $$
g_n:[0,1]\longrightarrow [0,1], \ \ n\in\N,
$$
of the Gauss map $G$ are given by the formulas 
$$
g_n([0,1])=\frac{1}{x+n}.
$$
and 
$$
g_n(x)=\left[\frac{1}{n+1},\frac 1 n\right].
$$
The collection of maps $$(g_n)_{n\in\mathbb N}$$ forms an infinite Iterated Function System satisfying the Open Set Condition.

It is natural to consider the subsystems $G_n$ consisting of $n$ initial maps $g_1,\dots g_n$. The limit set $J_n$ of the system $G_n$ is a Cantor set consisting of irrational numbers in $[0,1]$ having continued fraction expansion with entries bounded above by $n$.
Thus, the union $$\mathcal B=\bigcup_{n\in\mathbb N} J_n$$ is exactly the set of badly approximable numbers, i.e., 
$$
\mathcal B=\left\{x\in [0,1]\setminus \mathbb Q: \exists\, C\in(0,\infty)\colon \left|x-\frac p q\right|>\frac {C}{q^2} \quad \text{for all} \quad \frac p q\in\mathbb Q\right \}
$$
Denote by $h_n$ the Hausdorff dimension of the limit set $J_n$. A famous result of Doug Hensley, both in dynamics and number theory, is  (see \cite{hensley}) that
\begin{equation}\label{eq:asymp_dim for nonlinear Gauss}
\lim_{n\to\infty}(1-h_n) n=\frac{6}{\pi^2}.
\end{equation}
See also \cite{dfsu} for further estimates of the asymptotics of $h_n$ and \cite{ct} for discussion on the dimension of (more general) sets $\mathcal B(t)=\{x\in [0,1]\colon G^k(x)\ge t \quad\text{for all}\quad k\in\mathbb N\}$.

In the paper \cite{uz} a one level deeper question was asked. It concerned the behavior of the numerical values of $h_n$-dimensional Hausdorff measures $H_{h_n}(J_n)$ of the sets $J_n$ as $n\to\infty$. The system $G_n = \{g_j \}_{j=1}^n$ is a finite iterated function system of conformal maps, satisfying Open Set Condition. It follows from general theory of conformal iterated function systems that $0<H_{h_n}(J_n)<\infty$ (see, e.g. \cite{MU}, Theorem 4.2.11). The main result in \cite{uz} was that the function 
\[
\N\ni n\longmapsto H_{h_n}(J_n)
\]
has a limit as $n$ tends to infinity and, furthermore, this limit is equal to $1$. In formulas:
\beq
\label{320260226}
\lim_{n\to\infty} H_{h_n}(J_n)=1=H_1([0,1]).
\eeq
Thus, the function $\N\ni n\mapsto H_{h_n}(J_n)$ is continuous at infinity.

In the present paper we address and fully answer a further, natural, much subtler, and much more involved question about asymptotics of the function 
$$
\N\ni n\longmapsto H_{h_n}(J_n).
$$
In other words, we ask about the rate in which the values $H_{h_n}(J_n)$ converge to $H_1([0,1])=1$ when $n$ tends to infinity. Abbreviating $H_{h_n}(J_n)$ by $H_{n}$, we prove, as our main result stated in Theorem~\ref{t120260312}, that
$$
\lim_{n\to\infty}\frac{1-H_n}{(1-h_n)\ln n}=1
$$
and
$$
\lim_{n\to\infty} \frac{n(1-H_n)}{\ln n}= \frac{6}{\pi^2}.
$$
This result has a clear number theoretical and dynamical flavor. Our methods of proofs are dynamical (the theory of conformal iterated function systems and beyond), functional analytic (spectral theory and perturbations of linear operators), number theoretical (\cite{hensley}, \cite{uz}, and beyond), and we prove many hard auxiliary estimates. 

As a prelude to our treatment of the Gauss map, we deal in Part~1 with its piecewise linear version. This is interesting on its own and we substantially use some of the results obtained for this piecewise linear case in our proofs about the Gauss map itself.

We naturally replace the actual non-linear branches $g_n$ by linear ones, defining now
$$
g_n(x):=-\frac{1}{n(n+1)}x+\frac 1 n, \  \  n\in \N.
$$
Then, as in the Gauss map  
$$
g_n([0,1])= \left [\frac{1}{n+1}, \frac{1}{n} \right ] ,
$$
but these maps are now piecewise linear. We call this system the linear analogue of the Gauss map/system. We use the same symbols $g_n$, $J_n$, $h_n$, $H_n$, $H_{h_n}(J_n)$, and more, for both the non-linear Gauss map and its linear analogue. Since will never consider them simultaneously in Part~1, Part~2, or Part~3, this will not lead to confusion and misunderstandings. 

The subsystems $G_n$ consisting of $n$ initial maps $g_1,\dots g_n$ have the same meaning as for the initial non-linear version. 
Denoting by $J_n$ their limit sets and by $h_n$ the Hausdorff dimension of $J_n$, we have a similar asymptotics as in 
\cite{hensley}.
\begin{equation}\label{eq:asymp_dim for linear Gauss}
\lim_{n\to\infty}(1-h_n)\cdot n=\frac{1}{\chi}.
\end{equation}
where $\chi$ is the Lyapunov exponent of the system $(g_n)_{n\in\mathbb N}$ with respect to the Lebesgue measure (which is invariant), i.e.
$$\chi=\sum_{n=1}^\infty \frac{\log(n(n+1))}{n(n+1)}.$$
This result  can be found in \cite{dfsu}.
The Hausdorff measure $H_{h_n}(J_n)$ is again positive and finite for all $n\ge 2$ since the system $\{g_j\colon j\in\{1,\dots n\}\}$ is a systems of similarities satisfying the Open Set Condition.

As in \cite{uz}, we have also the following continuity result with a very similar proof.

\begin{equation}\label{eq:continuity_hausdorff}
\lim_{n\to\infty} H_{h_n}(J_n)=1=H_1([0,1]),
\end{equation}
See also \cite{tryniecki} for much more  general continuity results for  sequences of finite iterated function systems converging to an infinite one.

Our main result in the case of the linear analogue of the Gauss map is equally full (we prove it first) as for the non-linear Gauss map. It states, again abbreviating $H_{h_n}(J_n)$ by $H_n$, see Theorem~\ref{thm: wszystko razem}, that 
\[
\lim \limits_{n \to \infty } \frac{1-H_n}{1-h_n} \cdot \frac{1}{\ln n} = 1
\]
and 
\[
\lim \limits_{n \to \infty} \frac{n\cdot(1-H_n)}{\ln n}  = \frac{1}{\chi},
\]
where $\chi$ is the Lyapunov exponent of the system $G$ with respect to the Lebesgue measure.

\section{Preliminaries}

\subsection{The Gauss map and its piecewise linear version. Notation. }

Since the Gauss map and its linear analogue will be dealt with in separate sections, we, in line with Introduction, will use notation introduced below for both of these systems. As we have already said in Introduction, this will not lead confusion or misunderstanding. So, for both linear and non-linear case we introduce the following notation.

For every $k\in \mathbb N$ the function
$$
f_k:\left[\frac{1}{k+1},\frac{1}{k}\right]\longrightarrow [0,1] 
$$  
is defined either by the formula 
\[
f_k(x):=\left\{\frac 1 x\right \}
\]
if we consider the non-linear Gauss, or by the formula
$$
f_k(x):=-k(k+1)x+k+1.
$$
In either case each function $f_k$, $k\in \mathbb N$, is decreasing,
\[
f_k\left (\frac{1}{k+1}\right ) = 1, \  \  f_k\left (\frac{1}{k}\right ) = 0,
\]
and 
$$
f_k\left(\left[\frac{1}{k+1},\frac{1}{k}\right]\right)=[0,1].
$$
In the former case $f_k$ is real-analytic while in the latter case it is even affine. 

For every $k\in \mathbb N$ we denote by $g_k$ the inverse map 
$$
f_k^{-1}:[0,1]\longrightarrow\left[\frac{1}{k+1},\frac{1}{k}\right].
$$
As it was actually already indicated in Introduction, in the linear case we have
\[
g_k(x)=-\frac{1}{k(k+1)}x+\frac 1 k,
\]
while for the original Gauss map
$$
g_k(x)=\frac{1}{x+k}.
$$
The collection of maps 
$$
G:= \{g_n: n\in\mathbb N\}
$$ 
forms a (linear or nonlinear) iterated function system. In either case we use the same notation $G$.

\begin{defi}\label{def:IFS Sn}
For every $n\in \mathbb N$, the iterated function system $G_n$ is defined by limiting the the system $G$ to its initial $n$ maps. More precisely,
$$
G_n := \{g_k\}_{k=1}^{n}. 
$$
\end{defi}

\begin{notation}
In the following sections, we use notation
$$
b_k:=\frac{1}{k}=g_k(0), \quad  and \quad a_k:=b_k-b_{k+1} 
$$
for all $k\in\mathbb N$.
\end{notation}
With this notation
$$
g_k(0)=b_k, \  \  g_k(1)=b_{k+1},
$$
and
$$
g_k([0,1])=[b_{k+1}, b_k].
$$

\begin{notation} We denote by $\mathbb N^*$ the set of finite sequences with integer entries.
For every finite sequence  
$$
\omega=(\omega_1, \omega_2, \omega_3,\dots ,\omega_m)\in \N^*,
$$ 
we denote  
$$
g_\omega:=g_{\omega_1}\circ g_{i_2}\circ \dots \circ g_{\omega_m}:[0,1]\longrightarrow [0,1].
$$
\end{notation}

\begin{defi}[and notation]\label{def: generacje fnl}
Given $n\in\N$ and $l\in\N$, we denote by $\mathcal{F}^{n}_l$ the $l$-th generation of intervals generated by the iterated function system $G_n$. More precisely:
\[
  \mathcal{F}^n_l := \big \{ g_\omega([0,1]):\omega  \in \mathbb N^l \big \}
\]
We call them cylinder sets of order $l$. Similarly, we denote by $\mathcal F_l$ the $l$-th generation of intervals generated by the  system $G$:
\[
  \mathcal{F}_l := \big \{ g_\omega([0,1]): \omega  \in \mathbb N^l \big \}
\]
\end{defi}
Recall from the introduction that for each $n\in\N$, $J_n$ denotes the limit set of the iterated function system $G_n$. One of its possible definition is the following.
\[
J_n = \bigcap\limits_{k=1}^\infty \bigcup_{1\le \omega_1,\dots \omega_k\le n}g_{\omega_1}\circ g_{\omega_2} \circ \dots \circ g_{\omega_k}([0,1]).
\]
Then,
$$
J_n\subseteq [b_{n+1}, b_1]=[b_{n+1},1].
$$
Recall from the introduction that by $h_n$ we denote the Hausdorff dimension of the set $J_n$. If $A$ is a subset of a metric space and $h\in [0,+\infty)$, then
by $H_h(A)$ we denote the Hausdorff measure of the set A in dimension $h$. 
By $H_n$ we denote the Hausdorff measure of the set $J_n$ evaluated  at its dimension $h_n$.

\subsection{Density theorems for general Hausdorff measures}
In this section we collect some well-known general density theorems. 
We start with
the following density theorem for Hausdorff measures (see \cite{mattila}, p.91).

\begin{fact}\label{fact:density} 
If $X$ is a metric space with $h:=\dim_H(X)$ being its Hausdorff dimension is such that $H_{h}(X)<+\infty$, then
\begin{equation}\label{eq:density_general}
\lim_{r\to 0}\left (\sup\left\{\frac{H_{h}(F )}{{\diam}^{h}(F)}: x\in F\subset X,\
\overline{F}=F,\ \diam(F)\le r\right\}\right )=1.
\end{equation}
for $H_{h}$--a.e. $x\in X$.
\end{fact}
 As a corollary of this fact, we  get the following
fundamental theorem which was extensively explored in e.g. in \cite{olsen} and \cite{ath}.

\begin{theorem}\label{thm_density_formula}
Let $X$ be a metric space such that $0< H_{h}(X)<+\infty$, where
$h:=\dim_H(X)$. 
Denote by $H^1_h$ the normalized $h$-dimensional Hausdorff measure on $X$, i.e. 
$$
H^1_h:=H_{h}^{-1}(X)H_{h}.
$$ 
Then we have for $H_{h}$-a.e. $x\in X$ that
\begin{equation}\label{eq:limitformula}
H_{h}(X)=\lim_{r\to 0}\left (\inf\left\{\frac{{\diam}^{h}(F)}{H_{h}^{1}(F\cap X)}:
          x\in F,\quad \overline{F}=F,\quad \diam(F)\le r\right\}\right ).
\end{equation}
\end{theorem}

Obviously, we also have the following.

\begin{cor}\label{cor:ratio}
If $X$ is a subset of a Euclidean metric space $\mathbb R^d$ such that $0< H_{h}(X)<+\infty$, where $h:=\dim_H(X)$, then we have for $H_{h}$--a.e. $x\in X$ that
\begin{equation}\label{eq:limitformula2}
H_{h}(X)=\lim_{r\to 0}\inf\left\{\frac{{\diam}^{h}(F)}{H_{h}^{1}(F \cap X)}\right\},
\end{equation}
where, for every $r>0$, the infimum is taken over all closed convex sets $F\subset\mathbb R^d$ containing $x$ with $\diam(F)\le r$.

Equivalently:
\begin{equation}\label{eq:limitformula2B}
\frac{1}{H_{h}(X)}
=\lim_{r\to 0}\sup\left\{\frac{H_{h}^{1}(F \cap X)}{{\diam}^{h}(F)}\right\}
\end{equation}
where, for every $r>0$, the supremum  is taken over all closed convex sets $F\subset\mathbb R^d$ containing $x$ with $\diam(F)\le r$.
\end{cor}

For subset of the real line we have even the following simpler formula.

\begin{cor}\label{cor:fundamental}
If $X$ is a subset of an interval $\Delta\subset\mathbb R$ such that $0< H_{h}(X)<+\infty$, where $h:=\dim_H(X)$, then we have for $H_{h}$--a.e. $x\in X$ that
\begin{equation}\label{eq:density_interval-1}
H_{h}(X)=\lim_{r\to 0}\left (\inf\left\{\frac{{\diam}^{h}(F)}{H_{h}^{1}(F\cap X)}\right\}\right ),
\end{equation}
where, for every $r>0$, the infimum is taken over all closed
intervals $F\subset\Delta$ containing $x$ with $0<\diam(F)\le r$.

Equivalently:
\begin{equation}\label{eq:density_interval-1B}
\frac{1}{H_{h}(X)}
=\lim_{r\to 0}\left (\sup\left\{\frac{H_{h}^{1}(F \cap X)}{{\diam}^{h}(F)}\right\}\right ),
\end{equation}
where, for every $r>0$, the supremum is taken over all closed
intervals $F\subset\Delta$ containing $x$ with $0<\diam(F)\le r$.
\end{cor}


\subsection{Hausdorff measures and density theorems for the linear Gauss system}

For iterated function systems on the interval $[0,1]$ consisting of similarities, i.e. affine maps, and satisfying the Strong Separation Condition, the above results can be restated in an even more convenient form, as it was observed and used first in \cite{olsen}.

\begin{prop}\label{prop:density_ifs}
Let $S$ be an iterated function system consisting of contracting  similarities satisfying the Strong Separation Condition.
Let $J$ be the limit set of this system, and $h=\dim_H(J)$. Then 
\[
 \sup\limits \left \{ \frac{H_h(F \cap J)}{{\rm{diam}}^{h}(F)}: F\subset [0,1] \,  \text{ is a closed interval} \right \}  = 1,  
\]
\[
H_h(J)=\inf\limits \left \{ \frac{{\diam}^{h}(F)}{H_{h}^{1}(F\cap J)}: F\subset [0,1] \,  \text{ is a closed interval} \right \},  
\]
and
\[
\frac{1}{H_{h}(J)}= \sup\limits \left \{\frac{H_{h}^{1}(F \cap J)}{{\diam}^{h}(F)}: F\subset [0,1] \,  \text{ is a closed interval} \right \}.  
\]
\end{prop}
In particular, we have the following.

\begin{equation}\label{eq: gestosc liniowych}
\bal
H_{h_n}(J_n) 
&= \inf \left \{  \frac{{\rm{diam}}^{h_n}(F)}{H_{h_n}^1(F \cap J_n)}: F \subset [0,1] \,  \text{is a closed interval}\right\}
\\
and
\\
\frac{1}{H_{h_n}(J_n) }
&= \sup \left \{  \frac{{\rm{diam}}^{h_n}(F)}{H_{h_n}^1(F \cap J_n)}: F \subset [0,1] \,  \text{is a closed interval}\right\},
\eal
\end{equation}
where, we recall, $H_{h_n}^1$ denotes the normalized Hausdorff measure on $J_n$.

\begin{notation}
We will also use the symbol $m_n$ to denote the normalized 
$h_n$--dimensional Hausdorff measure on $J_n$, i.e.
$$
m_n:={H^1_{h_n}}\big|_{J_n}
$$
\end{notation}
Note that $m_n$ is the  $h_n$-{conformal measure} on $J_n$, i.e. 
it satisfies
$$m_n(g_j(A))=\int_A|g_j'|^{h_n}  dm_n$$
for every $j\le n$ and a  Borel subset $A\subset [0,1]$.


Recall that, to ease notation, we shall also write
\begin{equation}\label{eq:hausdorff_value_notation}
H_n:=H_{h_n}(J_n).
\end{equation}


\subsection{An auxiliary abstract result}

\begin{lemma}
\label{l120260303}
Let $t_n\in (0,1)$, $n\in\N$, be a sequence such that
\beq\label{120260210-03032026}
\limsup_{n\to\infty}n(1-t_n)<+\infty.
\eeq
If $0\le p\le q$ are integers such that $q-p+1\le n$ and $u_j\in [0,1]$, $p\le j\le q$ are such numbers that
$$
\sum_{j=p}^qu_j=1,
$$
then
$$
\frac{\sum_{j=p}^q u_j^{t_n}-1}{1-t_n}
\le \ln n+O\big((1-t_n)\ln^2 n \big ).
$$
\end{lemma}

\begin{proof}
Since the sum $\sum_{j=q}^l u_j^{t_n}$ attains its maximum when $u_j=\frac{1}{q-p+1}$, and since the number of summands does not exceed $n$, we get
$$
\bal
\sum_{j=p}^q u_j^{t_n}
&\le (q-p+1)(q-p+1)^{-t_n}
=(q-p+1)^{1-t_n}
\le n\left (\frac 1 n\right )^{t_n}
=e^{(1-t_n)\ln n}
\\
&=1+(1-t_n)\ln n+O\left(((1-t_n)\ln n)^2\right ).
\eal
$$
Consequently,
$$
\frac{\sum_{j=p}^q u_j^{t_n}-1}{1-t_n}
\le \ln n+O\big((1-t_n)\ln^2 n \big ).
$$
We are done.
\end{proof}

\vspace{20pt}

\part{Asymptotics of Hausdorff measure for piecewise linear analogue of the Gauss map.}\label{part:1}

\section{Abstract preparations}

We start with recalling the notion of the entropy of a finite partition. If $(X,\mu)$ is a probability space and $\mathcal A=\{A_1\dots A_m\}$ is a finite partition of the space $X$ into measurable sets of positive measure, then the entropy of the partition $\mathcal A$ is defined by the formula

$$H(\mathcal A):=-\sum_{j=1}^m \mu(A_j)\ln\mu(A_j).$$

Given three positive integers $k\le l\le n$, consider the probability space being the interval 
$$
[b_{l+1}, b_k]
$$ 
endowed with the normalized Lebesgue measure. Denote by $\mathcal P_{k,l}$ the partition of $[b_{l+1}, b_k]$ into $l-k+1$ subintervals:
$$[b_{l+1}, b_l], \dots  ,[b_{k+1}, b_k].$$
The normalized Lebesgue measure of each interval $[b_{j+1}, b_j]$, $1\le k\le j\le l$ is equal to
\beq\label{120260127}
w_j=w_j(k,l):=\frac{\frac 1 j-\frac{1}{j+1}}{\frac 1 k-\frac 1 {l+1}}.
\eeq
Note that then
\beq\label{220260127}
H(\mathcal P_{k,l})=-\sum_{j=k}^lw_j\log w_j.
\eeq

Let $t_n\in (0,1)$, $n\in\N$, be a sequence such that
\beq\label{120260210}
\limsup_{n\to\infty}n(1-t_n)<+\infty.
\eeq
As an immediate consequence of Lemma~\ref{l120260303}, we get that
\begin{equation}\label{eq:sum_powers-2}
\frac{\sum_{j=k}^l w_j^{t_n}-1}{1-t_n}
\le \ln n+O\big((1-t_n)\ln^2 n \big ).
\end{equation}

Denoting, as we will always do,
\beq\label{320260127}
\Delta_j:=[b_{j+1}, b_j]=\left[\frac{1}{j+1}, \frac 1 j\right], \  j\in \N,
\eeq
we can write
\beq\label{420260127}
w_j=w_j(k,l)
=\frac{\frac 1 j-\frac{1}{j+1}}{\frac 1 k-\frac 1 {l+1}}
=\frac{|\Delta_j|}{|[b_{k+1}, b_k]|}.
\eeq


For a real number $t$, we denote by $[t]$ the integer par of $t$, i.e. the largest integer $\le t$. Our first result in this section is the following.

\begin{lemma}\label{cor:large_entropy}
For each $n\in\mathbb N$ and $\varepsilon > 0$ consider the partition
 $\mathcal P_{k,l}$, where $k=k(n,\varepsilon)=[n-n^{1-\varepsilon}]+1$ and $l=l(n)=n$. Then 
 $$\liminf_{n\to\infty} \frac{H(\mathcal P_{k,l})}{\ln n}\ge 1-\varepsilon. $$
\end{lemma}

\begin{proof}
We have
$$
w_j=\frac{\frac 1 j-\frac{1}{j+1}}{\frac 1 k-\frac 1 n}=\frac{1}{j(j+1)}\cdot \frac{k}{n-k}\cdot n.
$$
Since the function $x\mapsto \frac{x}{n-x}$ is increasing on the interval $(0,n)$, and 
$$
n-n^{1-\varepsilon}
\le [n - n^{1-\varepsilon}]+1
\le n-n^{1-\varepsilon}+1,
$$
we therefore get the following bounds: 
\[
w_j \ge  \frac{\frac{1}{j}-\frac{1}{j+1}}{\frac{1}{ n - n^{1-\varepsilon}}-\frac{1}{n}} = \frac{\frac{1}{j(j+1)}}{\frac{n^{1-\varepsilon}}{n(n - n^{1-\varepsilon})}} = \frac{\frac{1}{j(j+1)}}{\frac{n^{-\varepsilon}}{n - n^{1-\varepsilon}}} = \frac{(n - n^{1-\varepsilon})n^{\varepsilon}}{j(j+1)}
\]
and 
$$
w_j\le\frac{(n-n^{1-\varepsilon}+1)n^\varepsilon}{j(j+1)}\cdot (1-n^{\varepsilon-1})^{-1}.
$$
Thus,
$$
\bal
H(&\mathcal P_{k,l})=-\sum \limits_{j= [n-n^{1-\varepsilon}]+1}^{n}w_j\ln{w_j} \ge
\\
&\ge -\sum \limits_{j = [n-n^{1-\varepsilon}]+1}^{n}w_j\ln{\frac{(n - n^{1-\varepsilon}+1)n^{\varepsilon}(1+n^{\varepsilon-1})^{-1}}{j(j+1)}} 
\\
&=-\sum \limits_{j = [n-n^{1-\varepsilon}]+1}^{n}w_j\left[\ln{((n - n^{1-\varepsilon}+1)n^{\varepsilon}})-\ln(1+n^{\varepsilon-1})- \ln{(j(j+1))}\right ] 
\\
&= -\ln{(n - n^{1-\varepsilon}+1)} - \ln{n^{\varepsilon}}-\ln(1+n^{\varepsilon-1})+ \sum \limits_{j = [n-n^{1-\varepsilon}]+1}^{n}w_j\ln{(j(j+1))} 
\\
&=-\ln{(n - n^{1-\varepsilon}+1)} - {\varepsilon}\ln{n}-\ln(1+n^{\varepsilon-1})+ \sum \limits_{j = [n-n^{1-\varepsilon}]+1}^{n}w_j\ln{(j(j+1))}.
\eal
$$
Now, focusing on the last term of this expression, by observing that
$$
\frac{\ln (j(j+1))}{j(j+1)}\ge \frac{\ln (n(n+1))}{n(n+1)}
$$
and noting that the number of summands in this last term is larger than or equal to $n^{1-\varepsilon}$, we get
\[
\bal
 \sum \limits_{j = [n-n^{1-\varepsilon}]+1}^{n}w_j\ln{(j(j+1))} 
&\ge \sum \limits_{j = [n-n^{1-\varepsilon}]+1}^{n}\frac{(n - n^{1-\varepsilon})n^{\varepsilon}}{j(j+1)}\ln{(j(j+1))} 
\\
&\geq (n - n^{1-\varepsilon})n^{\varepsilon} \cdot \frac{n^{1-\varepsilon}}{n(n+1)} \ln{(n(n+1))} 
\\
&= \frac{n - n^{1-\varepsilon}}{n+1} \ln{(n(n+1))}\cdot \frac{n^{1-\varepsilon}}{n^{1-\varepsilon}}.
\eal
\]
Putting the above estimates together, we obtain
\[
\bal
\frac{H(\mathcal P_{k,l})}{\ln{n}} 
&\geq \frac{-\ln{(n - n^{1-\varepsilon}+1)} - {\varepsilon}\ln{n}-\ln (1+n^{\varepsilon-1}) + \frac{n - n^{1-\varepsilon}}{n+1} \ln{(n(n+1))}}{\ln{n}}
\\
&= -\varepsilon - \left[ 1 + \frac{\ln{(1- \frac{1}{n^\varepsilon}+\frac{1}{n})}}{\ln{n}}\right]-\frac{\ln(1+n^{\varepsilon-1})}{\ln n}+ \frac{1 - n^{-\varepsilon}}{1+\frac{1}{n}} \left [\frac{\ln n+\ln (n+1)}{\ln n}\right ].
\eal
\]
So, taking the limit as $n \to \infty$, we obtain
\[
\liminf \limits_{n \to \infty} \frac{H(\mathcal P_{k,l})}{\ln{n}}  \geq 2-1-\varepsilon = 1-\varepsilon.
\]
The proof of Lemma~\ref{cor:large_entropy} is complete. 
\end{proof}

Sticking to the partitions $\mathcal P_{k,l}$ from Lemma~\ref{cor:large_entropy}, we shall prove the following.

\begin{lemma}
\label{l120260127}
If $s_n\in [0,1)$ for all $n\ge 1$ large enough,
$$
\lim_{n \to \infty}s_n=1,
$$
and $\varepsilon\in (0,1)$, then
\[
\varliminf\limits_{n \to \infty} \frac{1}{(1-s_n)\ln{n}}
\left(\frac{ \sum \limits_{j = [n-n^{1-\varepsilon}]+1}^{n} |\Delta_j|^{s_n}}{{\rm{diam}}^{s_n}\left(\bigcup_{j = [n-n^{1-\varepsilon}]+1}^{n} \Delta_j\right)}-1\right)
\ge 1-\varepsilon.
\]
\end{lemma}

\begin{proof}
Using the notation from \eqref{120260127}, \eqref{420260127}, and Lemma~\ref{cor:large_entropy}, we have
$$
\bal
L_n:&=\frac{ \sum \limits_{j = [n-n^{1-}]+1}^{n} |\Delta_j|^{s_n}}{{\rm{diam}}^{s_n}\left(\bigcup_{j = [n-n^{1-\varepsilon}]+1}^{n} \Delta_j\right)}-1
=\sum\limits_{j = k}^{n}w_j^{s_n}-1
\\
&=\sum\limits_{j = k}^{n}(w_j^{s_n}-w_j)
=\sum\limits_{j = k}^{n}w_j(w_j^{s_n-1}-1)
\\
&=\sum\limits_{j = k}^{n}w_j\left(e^{(s_n-1)\log w_j} - 1 \right).
\eal
$$
Noting also that $e^x-1\ge x$ for all $x\in\mathbb R$, we therefore obtain
$$
\frac{L_n}{1-s_n}
\ge \frac{1}{1-s_n}\sum_{j=k}^nw_j(1-s_n)(-\ln w_j)=H(\mathcal P_{k,n-1}).
$$
Applying now Lemma~\ref{cor:large_entropy}, we thus get
$$
\varliminf\limits_{n \to \infty} \frac{L_n}{(1-s_n)\ln{n}}
\ge 1-\varepsilon.
$$
The proof of Lemma~\ref{l120260127} is complete. 
\end{proof}

\section{Asymptotics of $H_n$: Lower bound.}\label{sec:linear_below}

We start with  a straightforward estimate for the Hausdorff measure of the sets $J_n$.  
\begin{prop}\label{lem: miara mniejsza niz 1}
Recall that $G$ is the piecewise linear analogue of the Gauss map. Then 
$$
H_n(J_n) \leq 1
$$ 
for all $n \in \N$, where, we recall, $J_n$ is the limit set of the truncated subsystem, introduced in Definition~\ref{def:IFS Sn}.
\end{prop}

\begin{proof}
Recall that $a_j = b_j - b_{j+1} = \frac{1}{j} - \frac{1}{j+1}$ for all $j\in\N$. For every $k\in\N$, consider the cover of $J_n$ by the cylinders of the $k$-th generation defined in Definition \ref{def: generacje fnl}, i.e. the elements of the cover $\mathcal F^n_k$.

There are $n^k$ intervals in this cover and
\[
\begin{aligned}
\sum\limits_{J\in \mathcal F^n_k } |J|^{h_n} 
&= \sum\limits_{1\le i_1 \dots i_k \leq n} \left(a_{i_1}\cdot \dots \cdot a_{i_k} \right)^{h_n} = \sum\limits_{1\le i_1 \dots i_k \leq n} a_{i_1}^{h_n}\cdot \dots \cdot a_{i_k} ^{h_n} 
\\
&= \left [a_{1}^{h_n} + \dots + a_{n} ^{h_n} \right ]^k = 1^k = 1.
\end{aligned}
\]
Since also $\lim_{k \to \infty}\sup\big\{{\rm{diam}}(J):J \in \mathcal F^n_k\big\}= 0$, the proposition thus follows.
\end{proof}

Equipped with Lemma~\ref{l120260127}, we can prove the following first estimate of the growth of the value $H_n$.

\begin{theorem}\label{thm: estimate from below}
\[
\liminf \limits_{n \to \infty }\frac{1-H_n}{1-h_n} \cdot \frac{1}{\ln n}\geq 1.
\]
\end{theorem}

\begin{proof} 
Using notation \eqref{120260127} (and \eqref{420260127}) and having any $\varepsilon\in (0,1)$, it follows from \eqref{eq: gestosc liniowych} that
$$
    \frac{1}{H_n}-1 = \sup  \left \{\frac{m_n(F)}{{\rm{diam}}^{h_n}(F)} :F\subset [0,1] \, \text{is a closed interval}\right \}-1
 \ge \sum\limits_{j = k(n,\varepsilon)}^{n}w_j^{h_n}-1.  
$$
Looking up at Proposition~\ref{lem: miara mniejsza niz 1} and \eqref{eq:asymp_dim for linear Gauss}, we can thus apply Lemma~\ref{l120260127} with $s_n=h_n$, $n\in \N$, to get
$$
\liminf \limits_{n \to \infty }\frac{\frac{1}{H_n}-1}{(1-h_n)\ln n}
\ge 1-\varepsilon.
$$
Letting $\varepsilon\searrow 0$, we thus get that
$$
\liminf \limits_{n \to \infty }\frac{\frac{1}{H_n}-1}{(1-h_n)\ln n}
\ge 1.
$$
Finally, invoking \eqref{eq:continuity_hausdorff}, we get
$$
\liminf \limits_{n \to \infty }\frac{1-H_n}{1-h_n} \cdot \frac{1}{\ln n}\ge 1.
$$
The proof of Theorem~\ref{thm: estimate from below} is complete.

\end{proof}
\section{Asymptotics of $H_n$: Upper bound }\label{sec:linear_above}

Now, we shall estimate from above of the ratio  $\frac{1-H_n}{1-h_n}$. Recall that
\[
0 \le \frac{1}{H_n} - 1=\frac{1-H_n}{H_n} = \sup\limits_{F}\left \{ \frac{m_n(F)}{{\rm{diam}}^{h_n}(F)}-1\right\},
\]
where supremum is taken over all closed intervals $F \subset [0,1]$. Obviously, we only need to consider the closed intervals $F$ that intersect $J_n$.
We proved in Theorem \ref{thm: estimate from below}~that
\[
\liminf\limits_{n \to \infty} \frac{1-H_n}{1-h_n} \cdot \frac{1}{\ln n} \geq 1.
\]
Moving on to estimating from above, our goal is to prove the following.

\begin{theorem}\label{thm: estimate from above}
\[
\limsup\limits_{n \to \infty} \frac{1-H_n}{1-h_n} \cdot \frac{1}{\ln n} \leq 1.
\]
\end{theorem}

The estimate from above is more involved  because now we have to estimate from above the supremum over all intervals.
Our proof will consist of several steps. We start with the simplest case of intervals of the form $[b_{l+1},b_{k}]$, $k \leq l$.

\

{\bf Step 1.} Estimates on the intervals  $[b_{l+1}, b_k]$.
We shall prove the following lemma. It is based on the formula \eqref{eq:sum_powers-2}.

\begin{lemma}\label{prop: 21}
\begin{equation}\label{eq: 15}
\sup\limits_{k\leq l \in \N} \left\{\frac{\frac{m_n([b_{l+1},b_k])}{{\rm{diam}}^{h_n}([b_{l+1},b_k])}-1}{1-h_n}\right\}
\le \ln n+O\left(\frac{\ln^2 n}{n} \right ).
\end{equation}
\end{lemma}
\begin{proof}

Fix two positive integers $l \geq k$ and put  
\begin{equation}\label{1202602010}
F := [b_{l+1},b_k].
\end{equation}
If $n \leq l$ then the set $[b_{l+1},b_{n+1}]\cap J_n$ consists at most of one point, whence $m_n([b_{l+1},b_{n+1}])=0$. Therefore, we  then get
\begin{equation}\label{2202602010}
\frac{m_n(F)}{{\rm{diam}}^{h_n}(F)} - 1 = \frac{m_n([b_{n+1},b_k])}{{\rm{diam}}^{h_n}(F)} - 1 \leq \frac{m_n([b_{n+1},b_k])}{{\rm{diam}}^{h_n}([b_{n+1},b_k])} - 1. 
\end{equation}
Thus, from now on we may and we will assume that $n\geq l$. By conformality of the measure $m_n$, we get 
$$
m_n([b_{l+1}, b_k]) = \sum \limits_{j=k}^{l}m_n([b_{j+1}, b_{j}])
\  \  and  \  \
m_n([b_{j+1}, b_{j}]) = |b_j-b_{j+1}|^{h_n}.
$$  
Keeping the notation $w_j = \frac{|[b_{j+1}, b_j]|}{|[b_{l+1}, b_k]|}$ for $j = k \dots l$, looking up at \eqref{eq:asymp_dim for linear Gauss}, and applying \eqref{eq:sum_powers-2}, we obtain
\begin{equation}\label{eq: 17'}
\frac{\frac{m_n([b_{l+1}, b_k])}{{\rm{diam}}^{h_n}([b_{l+1}, b_k])}-1}{1-h_n} 
= \frac{\sum\limits_{j = k}^{l} w_j^{h_n} - 1}{1-h_n} 
=\ln n+O\big((1-h_n)\ln^2 n \big )
=\ln n+O\left(\frac{\ln^2 n}{n} \right ).
\end{equation}
Along with \eqref{1202602010} and \eqref{2202602010}, this completes the proof of Lemma~\ref{prop: 21}, and simultaneously  the first step of the proof of Theorem~\ref{thm: estimate from above}.
\end{proof}

As a consequence of of this step, we almost immediately get the second step.

\

{\bf Step 2.} Estimate for the sets $F$ of the form $[0,b_k]$, $k\in\N$.
\begin{lemma}\label{prop: 0, bk}
\begin{equation}\label{eq: 0,bk}    
\sup\limits_{k \in \N}\left\{ \frac{\frac{m_n([0,b_k])}{({\rm{diam}}[0,b_k])^{h_n}}-1}{1-h_n}\right\}
 \leq \ln n+O\left(\frac{\ln^2 n}{n} \right ).
\end{equation}
\end{lemma}
\begin{proof}
Fix $k\in \N$ and set $ F := [0,b_k]$. If $n+1 \leq k$, then $[0,b_{k}]$ intersects $J_n$ at one point at most, whence $m_n([0,b_k]) = 0$. Therefore, the left-hand of \eqref{eq: 0,bk} is equal to $-1$, and thus we can assume that $k \le n$. 
But then $m_n([0,b_k])=m_n([b_{n+1},b_k])$, while $\diam([0,b_k])>\diam([b_{n+1}, b_k])$.
Invoking Lemma~\ref{prop: 21} ends the proof of Lemma~\ref{prop: 0, bk}.
\end{proof}

\par {\bf Step 3.} Estimates for intervals of the form $[0,r]$, $r \in (0,1]$.  
\begin{lemma}\label{lem: rozdzial przedzialu 0,r}
For every $r \in (0,1]$ the interval $[0,r]$ can be expressed as a union of adjacent closed intervals (for each $k$ the right endpoint of $I_k$ is the left endpoint of $I_{k+1}$):
\[
[0,r] = \bigcup_{m = 1}^{\infty}I_m.
\]
In this representation each interval $I_m$ is either  a union of some collection (finite or infinite) of intervals of $m$-th generation $\mathcal F^m$ or a ''degenerate'' interval  of the form $[b,b]$, where $b$ is an endpoint of some interval $F\in\mathcal F^m$. 
\end{lemma}
\begin{proof}
Take the largest k such that $b_k > r$. Then $b_{k+1} \leq r$. Put $I_1 = [0,b_{k+1}]$. If $b_{k+1}=r$ then the construction ends here, and $[0,r]= I_1$. Otherwise, we continue with the second step.
We have  $I_1 \subset [0,r]$, and $I_1 \cup [b_{k+1}, b_k] \not \subset [0,r]$.
The interval $$I_1':=[b_{k+1}, b_k)$$ is an infinite union of intervals
$$[b_{k+1,j}, b_{k+1, j+1}],\quad j=1,2,\dots $$
where we denoted:  $b_{k+1,j}:=g_k(b_j)=g_k\left (\frac{1}{j}\right )$. In particular, $b_{k+1,1}=b_{k+1}.$

Now denote $j_0 = \max \{ j \geq 1: b_{k+1, j} \leq r\}$ and put (see Figure \ref{fig: first and second generation intervals})
\[
I_2 = \bigcup\limits_{j = 1}^{j_0-1}[b_{k+1,j}, b_{k+1, j+1}]=[b_{k+1}, b_{k+1,j_0}].
\]
Notice that, in the case when $j_0=1$, the summation is void and  in this case $I_2 $ is a degenerate interval
$$I_2=[b_{k+1}, b_{k+1,1}]=[b_{k+1},b_{k+1}]=\{b_{k+1}\}.$$

\begin{figure}

\tikzset{every picture/.style={line width=0.75pt}} 

\begin{tikzpicture}[x=0.75pt,y=0.75pt,yscale=-1,xscale=1]

\draw    (58,29) -- (412,29) ;
\draw    (58,38) -- (58,19) ;
\draw    (411,39) -- (411,20) ;
\draw    (368,39) -- (368,20) ;
\draw    (299,39) -- (299,20) ;
\draw [color={rgb, 255:red, 208; green, 2; blue, 27 }  ,draw opacity=1 ]   (341,39) -- (341,20) ;
\draw    (182,39) -- (182,20) ;
\draw    (141,39) -- (141,20) ;
\draw    (299,39) -- (61.84,140.21) ;
\draw [shift={(60,141)}, rotate = 336.89] [color={rgb, 255:red, 0; green, 0; blue, 0 }  ][line width=0.75]    (10.93,-3.29) .. controls (6.95,-1.4) and (3.31,-0.3) .. (0,0) .. controls (3.31,0.3) and (6.95,1.4) .. (10.93,3.29)   ;
\draw    (60,151) -- (414,151) ;
\draw    (60,160) -- (60,141) ;
\draw    (413,161) -- (413,142) ;
\draw    (301,161) -- (301,142) ;
\draw [color={rgb, 255:red, 208; green, 2; blue, 27 }  ,draw opacity=1 ]   (231,161) -- (231,142) ;
\draw    (184,161) -- (184,142) ;
\draw    (109,161) -- (109,142) ;
\draw    (368,39) -- (412.2,140.17) ;
\draw [shift={(413,142)}, rotate = 246.4] [color={rgb, 255:red, 0; green, 0; blue, 0 }  ][line width=0.75]    (10.93,-3.29) .. controls (6.95,-1.4) and (3.31,-0.3) .. (0,0) .. controls (3.31,0.3) and (6.95,1.4) .. (10.93,3.29)   ;

\draw (53,45) node [anchor=north west][inner sep=0.75pt]   [align=left] {0};
\draw (406,46) node [anchor=north west][inner sep=0.75pt]   [align=left] {1};
\draw (337,46) node [anchor=north west][inner sep=0.75pt]  [color={rgb, 255:red, 208; green, 2; blue, 27 }  ,opacity=1 ] [align=left] {r};
\draw (357,44) node [anchor=north west][inner sep=0.75pt]   [align=left] {$\displaystyle b_{k}$};
\draw (289,44) node [anchor=north west][inner sep=0.75pt]   [align=left] {$\displaystyle b_{k+1}$};
\draw (232,42) node [anchor=north west][inner sep=0.75pt]   [align=left] {...};
\draw (173,47) node [anchor=north west][inner sep=0.75pt]   [align=left] {$\displaystyle b_{n}$};
\draw (135,46) node [anchor=north west][inner sep=0.75pt]   [align=left] {$\displaystyle b_{n+1}$};
\draw (227,166) node [anchor=north west][inner sep=0.75pt]  [color={rgb, 255:red, 208; green, 2; blue, 27 }  ,opacity=1 ] [align=left] {r};
\draw (143.78,169) node [anchor=north west][inner sep=0.75pt]   [align=left] {...};
\draw (175,168) node [anchor=north west][inner sep=0.75pt]   [align=left] {$\displaystyle b_{k+1,j_{0}}$};
\draw (95.46,168) node [anchor=north west][inner sep=0.75pt]   [align=left] {$\displaystyle b_{k+1,\ 2}$};
\draw (284,168) node [anchor=north west][inner sep=0.75pt]   [align=left] {$\displaystyle b_{k+1,j_{0} +1}$};
\draw (435,22) node [anchor=north west][inner sep=0.75pt]   [align=left] {First generation};
\draw (435,144) node [anchor=north west][inner sep=0.75pt]   [align=left] {Second generation};
\draw (5,168) node [anchor=north west][inner sep=0.75pt]   [align=left] {$\displaystyle b_{k+1} = b_{k+1, 1}$};
\draw (404,169) node [anchor=north west][inner sep=0.75pt]   [align=left] {$\displaystyle b_{k}$};

\end{tikzpicture}

    \caption{First and second generation intervals.}
    \label{fig: first and second generation intervals}
\end{figure}
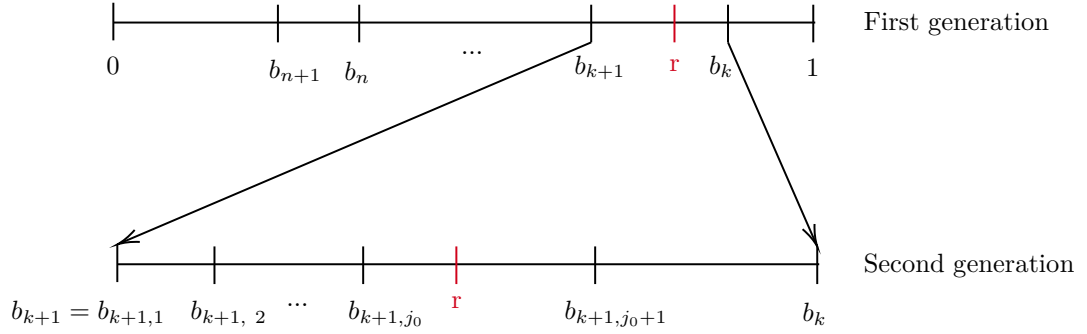

Again, if $r = b_{k+1, j_0}$ for some $j_0\in\mathbb N$ then the construction stops and the union in the representation \eqref{lem: rozdzial przedzialu 0,r} is just
\[
[0,r] = I_1 \cup I_2.
\]

Next, the inductive steps repeat the construction described above for $n=1$ and $n=2$. We describe it below in greater detail.
\par So, assume that the intervals $I_1, I_2,\dots,I_{n-1}$ are already defined so that the interval  $I_m$  is a union of some elements of $\mathcal F_m$ (perhaps degenerate, i.e. consisting of a single point)  the left endpoint of $I_m$ coincides with right endpoint of $I_{m-1}$ for all $m\le n-1$, and    
$$I_1 \cup \dots \cup I_{n-1}\subset [0,r] $$

If $[0,r] =I_1 \cup \dots \cup I_{n-1},$ then the construction ends here. Otherwise we proceed with the inductive step as follows.

Since $r$ is not the endpoint of $I_{n-1}$ and of any interval in $\mathcal{F}_m$, $m \leq n-1$,  the intervals of the collection $\mathcal F_{n-1}$ do not accumulate to the right of $r$.
Let $I_{n-1}' : = [c,d]$ be the interval belonging to the collection $\mathcal{F}_{n-1}$ and adjacent to $I_{n-1}$.
Because of the above remark, the interval  $I_{n-1}'$ is well defined.
By the construction, we have that $r \in (c,d)$. In the inductive step, define interval $I_n$ as the union:
\[
I_n = \bigcup\limits_{\substack{
        I \in \mathcal{F}_n \\
        I \subset I_{n-1}' \cap [0,r]}} I
\]
Note that
\begin{enumerate}
    \item the left endpoint of $I_n$ is  $c$,
    \item the above union is finite or infinite, depending on whether $n$ is even or odd. As in step for $n=2$, for even $n$ it may even happen that the above union of intervals  is ''degenerate'' just coincides with the left endpoint of the interval $I_{n-1}'$.
\end{enumerate}
\end{proof}

In this way, the interval $[0,r]$ is expressed as a union (finite or infinite) of the intervals $I_n$. As in previous steps, denote 
\begin{equation}\label{eq:wn}
w_k := \frac{|I_k|}{[0,r]},\quad  k = 1, 2, \dots.
\end{equation}
Then, obviously, $0\leq w_k \leq 1$ and
\[
\sum\limits_{k = 1}^{\infty} w_k = 1.
\]
Of course, it may happen that $w_k = 0$ for some  $k$; and it may happen that the summation is only over a finite number of indices $m$.

For further estimates of the sum $\sum w_i^h$ we need  the following.
\begin{lemma}\label{lem:even}
Assume that $n$ is even. Then $|I_{n}|\le |I_{n-1}|$, and , consequently, 
$$w_n\le w_{n-1}.$$
\end{lemma}
\begin{proof}
We have 
\begin{equation}\label{eq:scaling}
|b_{k}-b_{k+1}|\le b_{k+1}
\end{equation}
since, we recall, $b_k=\frac 1 k$, $k\in \N$. So, indeed, $|I_2|\le |I_1|$, since $I_2\subset [b_{k+1},b_k]$.
Now, notice that the intervals $I_{n-1}$ and $I_{n-1}'$ are affine copies of the intervals $[0,b_{k+1}]$ and $[b_{k+1},b_k]$ (for some  $k\in\mathbb N$).

Indeed,  both $I_{n-1}$ and $I_{n-1}'$ are contained in the same branch of injectiveness of $(n-2)$-th iterate of the linear Gauss map, and mapped by this branch onto $[0,b_{k+1}]$ and $[b_{k+1},b_k]$, respectively.
Since $I_n$ is a subset of $I_{n-1}'$, the lemma follows.
\end{proof}

\begin{lemma}\label{lem:odd}
Assume $n$ is odd. Then $|I_{n+2}|\le \frac{1}{4} |I_n|$ and, consequently, 
$$w_{n+2}\le \frac 1 4 w_n.$$
\end{lemma}
\begin{proof}
As in the proof of Lemma~\ref{lem:even}, we first look at the initial generations for $n=1$ and $n=3$. As we noticed, the set $I'_1$ is of the form $[b_{k+1}, b_k]$.  The set $I_3$ is a subset of the interval  $[b_{k+1,j_0}, b_{k+1,j_0+1}]$, and is a union of some intervals of third generation, contained in it, with at least one (the one most to the right) omitted. Since this last  omitted interval occupies half of the length of the interval $[b_{k+1,j_0}, b_{k+1,j_0+1}]$  , we have that

$$I_3=\bigcup_{m=m_0}^\infty [b_{k+1, j_0,m+1}, b_{k+1,j_0, m}],  $$ with some $m_0>1$, and thus
$$|I_3|\le \frac 1 2 \big|[b_{k+1,j_0}, b_{k+1,j_0+1}]\big|.$$ On the other hand, 
$$\big|[b_{k+1,j_0}, b_{k+1,j_0+1}]\big|\le \frac 1 2 |I'_1|\le \frac 1 2|I_1|$$
by \eqref{eq:scaling}.

In this way, we have the required estimate for $n=1$.
The general case  follows, as in Lemma~\ref{lem:even}, from the fact that $I_n$ and $I_{n+2}$ are affine copies of the  intervals $[0,b_{k+1}]$ and the appropriate interval of the form 
$$\bigcup_{m=m_0}^\infty [b_{k+1, j_0,m+1}, b_{k+1,j_0, m}],  $$
with some $m_0>1$, under the branch of $n-1$-th iterate of the linear Gauss map $G$. The proof of Lemma~\ref{lem:odd} is complete.
\end{proof}

\begin{figure}

\tikzset{every picture/.style={line width=0.75pt}} 

\begin{tikzpicture}[x=0.75pt,y=0.75pt,yscale=-1,xscale=1]

\draw    (158.83,73) -- (393,73) ;
\draw    (159,82) -- (159,63) ;
\draw    (512,83) -- (512,64) ;
\draw    (393,83) -- (393,64) ;
\draw    (279,83) -- (279,64) ;
\draw [color={rgb, 255:red, 208; green, 2; blue, 27 }  ,draw opacity=1 ]   (323,82) -- (323,63) ;
\draw    (279,64) .. controls (280.83,28.13) and (334.83,63.13) .. (334.83,34.13) ;
\draw    (393,64) .. controls (393.83,29.13) and (334.83,63.13) .. (334.83,34.13) ;
\draw    (158,63) .. controls (159.83,27.13) and (222.1,62.98) .. (222.1,33.98) ;
\draw    (279,64) .. controls (279.83,29.13) and (222.1,62.98) .. (222.1,33.98) ;

\draw (154,89) node [anchor=north west][inner sep=0.75pt]   [align=left] {0};
\draw (507,90) node [anchor=north west][inner sep=0.75pt]   [align=left] {1};
\draw (319,89) node [anchor=north west][inner sep=0.75pt]  [color={rgb, 255:red, 208; green, 2; blue, 27 }  ,opacity=1 ] [align=left] {r};
\draw (449,98) node [anchor=north west][inner sep=0.75pt]   [align=left] {...};
\draw (264,94) node [anchor=north west][inner sep=0.75pt]   [align=left] {$\displaystyle b_{k+1}$};
\draw (385,94) node [anchor=north west][inner sep=0.75pt]   [align=left] {$\displaystyle b_{k}$};
\draw (329,10.07) node [anchor=north west][inner sep=0.75pt]   [align=left] {$\displaystyle I'_{1}$};
\draw (218,10) node [anchor=north west][inner sep=0.75pt]   [align=left] {$\displaystyle I_{1}$};

\end{tikzpicture}

    \caption{$I_1'$ compared to $I_1$}
    \label{fig:i_1' vs bk+1,bk}
\end{figure}
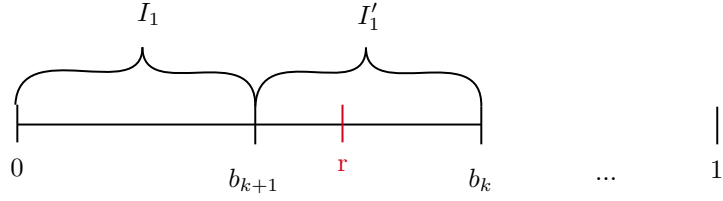

In order to estimate $\sum w_j^h$, we need the following simple lemma.
\begin{lemma}\label{lem: dokladne wyznaczenie alpha}
Consider all sequences $(x_j)_{j=1}^\infty$ such that 
$$
x_j\in [0,+\infty), \  \
\sum_{j=1}^\infty x_j = 1,
\  \ and \  \  x_{j+1} \leq \alpha x_j
$$ 
for some $\alpha \in (0,1)$ and all $j \geq 1$. Let $h\in (0,1)$.  Then the sum 
$$
\sum_{j = 1}^\infty x_j^h
$$ 
is finite and attains its maximum for the sequence 
$$
x_j := (1-\alpha)\cdot \alpha^{j-1}, \ j \in N,
$$ 
and its value is equal to 
$$
S_\alpha:= \frac{(1-\alpha)^h}{1-\alpha^h}.
$$
\end{lemma}

\begin{proof} The proof is elementary and we provide it for the sake of completeness and convenience of the reader. 

Let
\[
\bal
Y :&= \left\{x=(x_1, x_2 \dots) \in \mathbb{R}^\N: x_j \geq 0, \ \sum\limits_{i=1}^\infty x_i = 1, \,  \text{ and } \,  x_{j+1} \leq \alpha x_j  \,  \text{ for all  } \,  j\in \N \right\}
\\
&\subseteq [0,1]^{\N}.
\eal
\]
The set $Y$ is convex. We shall prove that equipped with the distance 
$$
d(x,x') := \sup_{j\ge 1}\big\{ |x_j-x_j'|\big\}.
$$
the set $Y$ is compact. Indeed, it immediately follows from the definition of $Y$ that
\beq
\label{120260115}
x_j\le \alpha ^{j-1}
\eeq
for every $x\in Y$. Consequently,
\beq
\label{120260120_1}
\sum\limits_{i=1}^j x_i\ge 1-(1-\alpha)^{-1}\alpha ^{j}
\eeq
\beq
\label{220260115}
\max_{1\le i\le j}\big\{ |x_j-x_j'|\big\}+2\alpha ^j
\le d(x,x') 
\le \max_{1\le i\le j}\big\{ |x_j-x_j'|\big\}+2\alpha ^j
\eeq
for every $j\in \N$ and all $x, x'\in Y$. 

Now let $\big(x^{(n)}\big)_{n=1}^\infty$ be a sequence of elements of $Y$. Since the space $[0,1]^{\N}$ endowed with the product topology is compact, this sequence contains a subsequence $\big(x^{(n_k)}\big)_{k=1}^\infty$ converging in the product topology. Denote its limit by $z$. This convergence means that
\beq
\label{320260115}
\lim_{k\to\infty}x_j^{(n_k)}=z_j
\eeq
for every $j\in \N$. It then immediately follows that
\beq
\label{420260115}
z_j \geq 0, \  \sum\limits_{i=1}^\infty z_i \le 1,  \   z_{j+1} \leq \alpha z_j \   \text{ and } \  \sum\limits_{i=1}^j z_i\ge 1-(1-\alpha)^{-1}\alpha ^{j}
\eeq
for all $j\in \N$. Consequently, $\sum_{i=1}^\infty z_i \ge 1$, whence furthermore,$\sum\limits_{i=1}^\infty z_i = 1$.
Thus
\beq
\label{120260120}
z\in Y.
\eeq
Now, fix an $\varepsilon>0$. Take then $j\in \N$ so large that
\beq
\label{220260120}
\alpha ^{j}<\varepsilon/4.
\eeq
Next, for every $i=1, 2, \ldots, j$ taken $N_i$ so large that 
$$
\big|x_i^{(n_k)}-z_j\big|<\varepsilon/2
$$
for all $k\ge N_j$. It follows from this, \eqref{220260120}, and \eqref{220260115} that
$$
d\big(x^{(n_k)},z\big)<\varepsilon/2+\varepsilon/2=\varepsilon
$$
for all $k\ge \max\big\{N_i:1\le i\le j\big\}$. Hence,
$$
\lim_{k\to\infty}x^{(n_k)}=z\in Y
$$
in the topology on $Y$ generated by the metric $d$. Thus, the space $(Y,d)$ is compact.

Moreover, it immediately follows from \eqref{220260115} that the topology on $Y$ generated by the metric $d$ and the topology on $Y$ inherited from the product topology on $[0,1]^{\N}$ are the same. Therefore, the function
$$
Y\ni x\mapsto x_j^h\in \mathbb R
$$
is continuous on $(Y,d)$ for every $j\in \N$. Thus the function
$\Phi:Y\rightarrow \mathbb R$, given by the formula 
$$
\Phi(x):= \sum_{j = 1}^\infty x_j^h\in \mathbb R,
$$
is well defined and it is continuous by virtue the Weierstrass $M$ test and \eqref{120260115}.

Moreover, the function $\Phi$ is  strictly concave. Thus, there exists a unique point $\hat{x} \in Y$ at which the function $\Phi$ attains its maximum $S$. We have
\[
S = \Phi\big((\hat{x}_j)_{j=1}^\infty)\big) 
= \hat x_1^h + \Phi\big((\hat{x}_j)_{j=2}^\infty)\big)
\]
Because of maximality, we have
\begin{equation*}
\begin{aligned}
\Phi(\hat{x})
&= \sup \Big\{\Psi\big((y_j)_{j=1}^\infty)\big)\colon  y_i \geq 0, \ y_{i+1} \leq \alpha y_i, \text{ and } \sum_{j = 1}^\infty y_j = 1 - \hat x_1 \  \forall \ i\in\N \Big\}
\\
&= \sup\limits_{y\in Y} \big\{(1-\hat{x}_1)^h \Phi(y)\big\}  
= (1-\hat{x}_1)^h \cdot S.
\end{aligned}
\end{equation*}
In particular, this implies that
\[
\Phi\left (\left (\frac{1}{1-\hat{x}_1} \hat{x}_2, \frac{1}{1-\hat{x}_1} \hat{x}_3, \dots \right) \right) = S.
\]
So, by uniqueness, we conclude that
\[
\hat{x}_{j+1} = \hat{x}_j(1-\hat x_1)
\]
for all $j\in\N$. This implies that 
$$
\hat{x}_j =\hat{x}_1 \cdot \left({1-\hat{x}_1}\right)^{j-1}
$$
for all $j\in\N$. Hence,
$$
\hat{x}_1\in (0,1).
$$
Putting $\gamma:=1-\hat{x}_1$, we can write 
$$
\hat{x}_j = (1-\gamma) \gamma^{j-1}.
$$
Hence, $\gamma \leq \alpha$ and
\[
\sum \limits_{j = 1}^\infty \hat{x}_j^h = \frac{(1-\gamma)^h}{1-\gamma^h} = S_\gamma.
\]
Since $S_\gamma \leq S_\alpha$ and the sequence $(\hat{x_j})_{j = 1}^\infty$ is maximal, we have $\gamma = \alpha$ and 
$$
\hat{x}_j = (1-\alpha)\cdot \alpha^{j-1}
$$
for all $j\in\N$.
\end{proof}

\begin{lemma}\label{lem:est_wn_2}
Let $w_j$, $j\in\N$, be defined as in \eqref{eq:wn}. Then 
$$
\sum \limits_{j = 1}^\infty w_j^h  
\le 2^{1-h} \frac{(1-(1/2))^h}{1-(1/2)^h}
=\frac{2^{1-h}}{2^{h}-1},
$$
for every $h\in (0,1)$.
\end{lemma}
\begin{proof}

We have 
$$w_{2n+1}+w_{2n+2}\le \frac 1 2 (w_{2n-1}+w_{2n})$$
for all $n\ge1$.

Putting now 
$$
x_j:=w_{2j-1}+w_{2j}, \quad  j\in\N,
$$
we see that
$$\sum_{i=1}^\infty x_i=1 \quad\text {and}\quad x_{j+1}\le \frac 1 2 x_j.$$
for all $j\in\N$.
Therefore, the sequence $\left ({x_j}\right )_{j\in\mathbb N}$ satisfies the hypotheses of Lemma~\ref{lem: dokladne wyznaczenie alpha} with $\alpha=\frac 1 2$, and therefore
\[
\sum \limits_{j = 1}^\infty x_j^h  \le \frac{(1-(1/2))^h}{1-(1/2)^h}.
\]
But, since $x_j=w_{2j-1}+w_{2j}$ and the function $\varphi(x)=x^h$ is concave in $(0,\infty)$ $\left(\text{whence } \frac{1}{2}\left(\varphi(a)+\varphi(b)\right )\le \varphi(\frac{a+b}{2})\right)$, we thus get 
$$w_{2j-1}^h+w_{2j}^h\le 2^{1-h}(w_{2j-1}+w_{2j})^h=2^{1-h}x_j^h.$$
Lemma~\ref{lem:est_wn_2} is proved.
\end{proof}





\begin{lemma}\label{prop: 0r ograniczenie z gory} For every $r>0$ the following holds uniformly with respect to $r\in (0,1)$. 
\begin{equation}\label{eq:0r-1}
\begin{aligned}
\frac{m_n([0,r])}{{\rm{diam}}^{h_n}([0,r])}
&\leq  1 + (1-h_n)\ln n +o\big((1-h_n)\ln n\big).
\end{aligned}
\end{equation}
\end{lemma}

\begin{proof}


\par In order to make the calculation more transparent, we start with special case when the interval $[0,r]$  can be written in the representation from Lemma~\ref{lem: rozdzial przedzialu 0,r}. as the union of two intervals $I_1$, $I_2$. So, let $[0,r] = [0,b_{k+1}] \cup [b_{k+1},b_{k+1, j}]$. Then
\begin{equation}\label{eq: 22}
\frac{m_n([0,r])}{{\rm{diam}}^{h_n}([0,r])} = \frac{\overbrace{m_n([0,b_{k+1}])}^{A_k} +|b_k - b_{k+1}|^{h_n} \cdot \overbrace{m_n([b_j, 1])}^{B_j}}{{\rm{diam}}^{h_n}([0,r])}.
\end{equation}
First, we shall estimate the expression
\[
\frac{m_n([0,r])}{\underbrace{|b_{k+1}-0|^{h_n}}_{\widetilde{A_k}} + |b_k-b_{k+1}|^{h_n}\underbrace{(1-b_j)^{h_n}}_{\widetilde{B_j}}}
\]
Note that the expression in the denominator is not equal to $({\rm{diam}}([0,r]))^{h_n}$. 
\\We have by steps (1) and (2) the following estimates:
\[
\frac{\frac{A_k}{\widetilde{A_k}}-1}{1-h_n} 
\leq \ln n(1+Dn^{-1}\ln n)
\]
and
\[
\frac{\frac{B_j}{\widetilde{B_j}}-1}{1-h_n}
\leq \ln (1+Dn^{-1}\ln n).
\]
where $D\in (0,+\infty)$ is a constant independent on $n$. So, 
\[  
\bal
A_k + |b_k-&b_{k+1}|^{h_n} B_j \le
\\
&\leq \widetilde{A_k} + (1-h_n)\ln n(1+Dn^{-1}\ln n) \widetilde{A_k}  + 
\\
&\hspace{4mm} + |b_k - b_{k+1}|^{h_n}\big( \widetilde{B_j} + (1-h_n)\ln n(1+Dn^{-1}\ln n) \widetilde{B_j\big)}
\\
&=\widetilde{A_k}\Big(1 + (1-h_n)\ln n\big(1 + Dn^{-1} \ln n\big)\Big)
+ 
\\
&\hspace{4mm} + |b_k - b_{k+1}|^{h_n}\widetilde{B_j}\Big(1 + (1-h_n)\ln n\big(1 + Dn^{-1} \ln n\big)\Big).
\eal
\]
Thus,
\begin{equation}\label{eq: 23}
\frac{A_k +|b_k - b_{k+1}|^{h_n}B_j}{\widetilde{A_k} +|b_k - b_{k+1}|^{h_n}\widetilde{B_j}} 
\leq 1 + (1-h_n)\ln n\big(1 + Dn^{-1} \ln n\big),
\end{equation}
and, consequently,
\[
\frac{\frac{A_k +|b_k - b_{k+1}|^{h_n}B_j}{\widetilde{A_k} +|b_k - b_{k+1}|^{h_n}\widetilde{B_j}} -1}{1-h_n} 
\leq \ln n\big(1 + Dn^{-1} \ln n\big).
\]
Note, however that we need to estimate from above another ratio, namely \eqref{eq: 22}. In order to do that, using \eqref{eq: 23}, we  write
\beq
\label{320260210}
\bal
\frac{A_k + |b_k - b_{k+1}|^{h_n} B_j}{|b_{{k+1},j}- 0 |^{h_n}} 
&= \frac{A_k + |b_k - b_{k+1}|^{h_n} B_j}{\widetilde{A_k} + |b_k - b_{k+1}|^{h_n} \widetilde{B_j}} \cdot \frac{\overbrace{\widetilde{A_k} + |b_k - b_{k+1}|^{h_n} \widetilde{B_j}}^{R_n}}{|b_{k,j}- 0 |^{h_n}} 
\\
&\leq \Big(1 + (1-h_n)\ln n\big(1 + Dn^{-1} \ln n\big) \Big) R_n.
\eal
\eeq
Thus,  we need to estimate the "error term"  $R_n$. A straightforward  estimate (however  insufficient for the general case, which we shall discuss below) is provided in the following  simple claim.

\

Claim~1.
\[
R_n \leq 1 + O(n^{-1}).
\]

\begin{proof} Using \eqref{eq:asymp_dim for linear Gauss}, we get
\[
\bal
R_n &= \frac{|I_1|^{h_n} + |I_2|^{h_n}}{(|I_1|_1 + |I_2|)^{h_n}} = w_1^{h_n} + w_2^{h_n}\leq 2 \cdot \left (\frac{1}{2} \right)^{h_n} = 2^{1-h_n}
\\
&= e^{\ln 2 (1-h_n)} = 1 + (1-h_n)\ln 2 + O((1-h_n)^2)
\\
&= 1 + O(n^{-1}),
\eal
\]
where $w_1 = \frac{|I_1|}{|I_1|+|I_2|}$ and $w_2 = \frac{|I_2|}{|I_1|+|I_2|}$. The proof of Claim~1 is complete.
\end{proof}
Using also \eqref{eq: 22}, \eqref{320260210}, and \eqref{eq:asymp_dim for linear Gauss} again, we get
\bea
\label{120260212}
\bal
\frac{m_n([0,r])}{{\rm{diam}}^{h_n}([0,r])} 
&\leq \Big(1 + (1-h_n)\ln n\big(1 + Dn^{-1} \ln n\big) \Big)\big(1 + O(n^{-1})\big)
\\
&=1 + (1-h_n)\ln n +o\big((1-h_n)\ln n\big).
\eal
\eea
The above calculation might  suggest that passing to the general case, where the interval $[0,r]$ is represented as the infinite union of the subintervals $I_n$, as in Lemma~\ref{lem: rozdzial przedzialu 0,r},  would increase  the factor  $R_n$ which appears in the estimates. However, due to Lemma~\ref{lem:est_wn_2}, the error term $R_n$ is bounded by a constant and better.
Indeed, in the general case, i.e. for the decompositon described in Lemma~\ref{lem: rozdzial przedzialu 0,r}, exactly as in the formula \eqref{120260212}, we obtain
\beq
\label{420260210}
\frac{m_n([0,r])}{\sum\limits_{j=0}^\infty{\rm{diam}}^{h_n}(I_j)} \leq 1 + (1-h_n)\ln n +o\big((1-h_n)\ln n\big)
\eeq
and
\begin{equation}\label{eq: 31}
\frac{m_n([0,r])}{{\rm{diam}}^{h_n}([0,r])} = \frac{m_n([0,r])}{\sum\limits_{j=0}^\infty{\rm{diam}}^{h_n}(I_j)} \cdot R_n,
\end{equation}
where now $R_n$  (the ''error term'') is the infinite sum 
\begin{equation}\label{eq:Rn}
R_n = \sum \limits_{j = 0}^\infty w_j^{h_n}, \quad  w_j = \frac{|I_j|}{|[0,r]|}.
\end{equation}
Using Lemma~\ref{lem:est_wn_2}, we get
$$
R_n\le\frac{2^{1-h_n}}{2^{h_n}-1}.
$$
Consider the function $g:(-\infty,1)\longrightarrow (0,+\infty)$ given by the formula
$$
g(x):= \frac{2^x}{2^{1-x}-1}.
$$
This function is analytic and a direct calculation shows that
$$
g(0)=1 \  \  {\rm and } \  \  g'(0)=3\ln 2.
$$
Therefore,
$$
g(x)=1+O(x).
$$
Hence, using \eqref{eq:asymp_dim for linear Gauss} again, we get
$$
R_n\le \frac{2^{1-h_n}}{2^{h_n}-1}
=1+O(1-h_n)
=1+ O(n^{-1}).
$$
Inserting this to \eqref{eq: 31} and using \eqref{420260210}, we get
$$
\bal
\frac{m_n([0,r])}{{\rm{diam}}^{h_n}([0,r])} 
&\le \big(1 + (1-h_n)\ln n +o\big((1-h_n)\ln n\big)\Big) \cdot R_n
\\
&\le \Big(1 + (1-h_n)\ln n +o\big((1-h_n)\ln n\big)\Big)\big(1+ O(n^{-1})\big)
\\
&=1 + (1-h_n)\ln n +o\big((1-h_n)\ln n\big).
\eal
$$
The proof of Lemma~\ref{prop: 0r ograniczenie z gory} is complete.
\end{proof}

The estimate proved in Lemma~\ref{prop: 0r ograniczenie z gory}  immediately  give the following. 

\begin{lemma}\label{prop:0r}
\[
\limsup\limits_{n \to \infty} \left (\sup\limits_{r \in (0,1)}  \left\{\frac{\frac{m_n([0,r])}{{\rm{diam}}^{h_n}([0,r])}-1}{\left(1-h_n\right) \cdot \ln n} \right\} \right )\leq 1.
\]
\end{lemma}

{\bf Step 5.} In this step, we provide the estimate for  intervals of the form  $F = [r,1]$, $r\in (0,1)$. We shall prove the following.
\begin{lemma}\label{prop: r 1}
\[
\limsup \limits_{n\to\infty}\left (\sup\limits_{r \in (0,1) } \left\{\frac{\frac{m_n([r,1])}{{\rm{diam}}^{h_n}([r,1])}-1}{(1-h_n)\cdot \ln n}\right\}\right )\leq 1.
\]
\end{lemma}

\begin{proof}
Fix $r\in (0.1)$. We can assume that $\#([r,1]\cap J_n) > 1$, because otherwise the interval $[r,1]$ intersects the limit set in at most one point, and thus, $m_n([r,1]) = 0$. Then, we split our consideration into two cases:
\begin{itemize}
    \item[Case 1.] $r \geq \frac{1}{2}$ 
    
    \hspace{-5mm} or
    
    \item[Case 2.] $r < \frac{1}{2}$.
\end{itemize}
In  Case 1, applying the map $g_1^{-1}$ to the interval $[r,1]$ transforms it to the interval $[0,r']$ where $r' = g_1^{-1}(r)$. In addition, $r' > b_{n+1}$, since $\#([r,1]\cap J_n) > 1$. Hence,
\[
\frac{\frac{m_n([r,1])}{{\rm{diam}}^{h_n}([r,1])}-1}{1-h_n} = \frac{\frac{m_n([0,r'])}{{\rm{diam}}^{h_n}([0,r'])}-1}{1-h_n}.
\]
So, we are done by applying Lemma \ref{prop:0r}.

In Case 2, we can find $l \geq 1$ such that $b_{l+1} \leq r < b_l$. If $l \geq n+1$, then we can set $r: = b_{n+1}$, which does not modify the value of $m_n([r,1])$, while ${\rm{diam}}[b_{n+1},1] \leq {\rm{diam}} [0,r]$. Then 
\[
\frac{m_n([b_{n+1},1])}{{\rm{diam}}([b_{n+1},1])} \geq \frac{m_n([r,1])}{{\rm{diam}}([r,1])}.
\]
So, from now on we can and we do assume that $l \le n$. Notice that then the interval $[r,1]$ can be represented as the union of two intervals
\[
[r,1] = I_1 \cup [b_l,1],
\]
where $I_1 \subset [b_{l+1}, b_l]$. Now, the interval $I_1$ is mapped by the branch $g_l^{-1}(I_1)=[0,r']$ with some $r' >0$ and
\[
\frac{m_n(I_1)}{m_n([b_{l+1}, b_l])} 
= \frac{m_n([0,r'])}{1}
=m_n([0,r']), \  \
\frac{{\rm{diam}}^{h_n}(I_1)}{{\rm{diam}}^{h_n}([b_{l+1}, b_l])} ={\rm{diam}}^{h_n}([0,r']) 
\]
and 
$$
m_n([b_{l+1},b_l]) = \diam^{h_n}([b_{l+1}, b_l]).
$$
Thus, using also \eqref{eq:0r-1} in Lemma~\ref{prop: 0r ograniczenie z gory}, we get
\begin{equation}
\begin{aligned}
&\frac{m_n(I_1)}{\diam^{h_n}(I_1)} 
= \frac{m_n([0,r'])}{\diam^{h_n}([0,r'])} \le
\\ 
& \hspace{5mm}\le 1 + (1-h_n)\ln n +o\big((1-h_n)\ln n\big).
\end{aligned}
\end{equation}
Similarly, Lemma~\ref{prop: 21} gives the same estimate for the interval $[b_l,1]$:
$$  
\frac{m_n([b_{l},1])}{{\rm{diam}}^{h_n}([b_{l},1])}
\le 1 + (1-h_n)\ln n +o\big((1-h_n)\ln n\big).
$$

Now, using the same estimate as in Claim~1 of the proof of Lemma~\ref{prop: 0r ograniczenie z gory}, we obtain that
\begin{equation}\label{eq:r1}
\begin{aligned}
\frac{m_n([r,1])}{\diam^{h_n}([r,1])}
&\le \big(1 + (1-h_n)\ln n +o\big((1-h_n)\ln n\big)\Big) \cdot R_n
\\
&\le \Big(1 + (1-h_n)\ln n +o\big((1-h_n)\ln n\big)\Big)\big(1+ O(n^{-1})\big)
\\
&=1 + (1-h_n)\ln n +o\big((1-h_n)\ln n\big).
\end{aligned}
\end{equation}

\end{proof}

{\bf Step 6.} The final step in the proof of  Theorem~\ref{thm: estimate from above} is to obtain an appropriate estimate for an arbitrary interval $F \subset [0,1]$.
We shall prove 
\begin{lemma}\label{prop: dowolny przedzial}
\[
\limsup \limits_{n\to\infty}\left (\sup \left\{ \frac{\frac{m_n(F)}{({\rm{diam}}(F))^{h_n}}-1}{(1-h_n)\ln n}\right\}\right )\leq 1,
\]
where the supremum is for every $n\in \N$ taken over all closed intervals $F\subset[0,1]$.
\end{lemma}

\begin{proof}
Fix $n \in \N$. Let $F \subset [0,1]$ be a closed interval. We can assume that $F\subset [b_{n+1}, 1]$ since otherwise we can consider the interval $\widetilde{F} = F \cap [b_{n+1},1]$ which does not change the value of $m_n(F)$, while ${\rm{diam}}(\widetilde{F}) \leq {\rm{diam}} (F)$, so that
\[
\frac{m_n(\widetilde{F})}{{\rm{diam}}(\widetilde{F})} \geq \frac{m_n(F)}{{\rm{diam}}(F)}.
\]
Now consider two cases:

Case~1. $F$ contains some basic interval $[b_{j+1},b_{j}]$ with some $j \leq n $

\hspace{1cm} or 

Case ~2. $F$ does not contain any interval of the form $[b_{j+1},b_{j}]$.
   
First, we will focus on Case~1. Let $[b_l, b_k]$ $k < l \leq n+1$, be the union of all basic intervals which are contained in $F$. Then $F$ can be represented as a union of three intervals:
\[
F = I_1 \cup [b_l, b_k] \cup I_2,
\]
where $I_1\subset [b_{l+1}, b_l]$ and $I_2\subset [b_k, b_{k-1}]$.

We get from the estimate \eqref{eq: 15} in  Lemma~\ref{prop: 21}, that
$$\frac{m_n([b_{l+1},b_k])}{{\rm{diam}}^{h_n}([b_{l+1}, b_k])}
\le 1 + (1-h_n)\ln n +o\big((1-h_n)\ln n\big).
$$

Since $g_l^{-1}(I_1)=[0,r]$ with some $r >0$, we get by Proposition~\ref{prop:0r}, that 
\begin{equation}
\frac{m_n(I_1)}{\diam^{h_n}(I_1)} =
\frac{m_n([0,r])}{\diam^{h_n}([0,r])}
\le 1 + (1-h_n)\ln n +o\big((1-h_n)\ln n\big).
\end{equation}

Similarly, $g_k^{-1}(I_2)=[r',1]$ with some $r'\in [0,1]$. So,
\[
\frac{m_n(I_2)}{\diam ^{h_n}(I_2)} = \frac{m_n([r', 1])}{\diam^{h_n} ([r', 1])}.
\]
Thus, \eqref{eq:r1} holds with  $[r',1]$ replaced by $I_2$ by virtue of Proposition \ref{prop: r 1}.

Finally, we repeat almost the same calculation as in in the proof of formula \eqref{eq:r1}.
The only difference now is that in the estimate of the "error term"  there are now three summands, so that, using Claim~1 from the proof of Lemma~\ref{prop: 0r ograniczenie z gory} for three summands, we get that

\begin{equation}\label{eq:ratio_weaker}
\begin{aligned}
\frac{m_n(F)}{\diam^{h_n}(F)}
&\le \Big(1 + (1-h_n)\ln n +o\big((1-h_n)\ln n\big)\Big)\big(1+ O(n^{-1})\big)
\\
&=1 + (1-h_n)\ln n +o\big((1-h_n)\ln n\big).
\end{aligned}
\end{equation}


Now, we move on to Case~2. The interval $F$ does not contain any interval of the form $[b_l, b_{k}]$. Then, again there are two subcases

Case~2a. either $F \subset [b_{k+1}, b_{k}]$ for some $k \leq n$ 

\hspace{1cm} or 

Case~2b. $F = I_1 \cup I_2 $ where $I_1 \subsetneqq [b_k, b_{k-1}]$ and $I_2 \subsetneqq [b_{k-1}, b_{k-2}]$. \label{subcase: subcase 2b}

In Case~2b the same way of estimate applies as in Case~1. The only difference is that now, there are two summands instead of three, so we obtain also the estimate:
$$
\frac{m_n(F)}{\diam^{h_n} (F)}
\le 1 + (1-h_n)\ln n +o\big((1-h_n)\ln n\big).
$$


In case Case~2a we proceed as follows. Since
 $F \subset [b_{k+1}, b_{k}] \in \mathcal{F}^n_1$ (intervals in the first generation of the construction of $J_n$), we can meaningfully define the number 
\[
M := \max \big \{ m \geq 1: F \subset  I \text{ for some } I \in \mathcal F^n_m \big \}.
\]
Then $F=g_\omega(F')$, where $\omega\in \{1, 2,\ldots, n\}^M$ and $F'\subset [0,1]$ is an interval such that $b_k\in {\rm int}(F')$
for some $k\le n$. 
Now, either $F'$ falls either into Case~1 or into Case~2b. Since the ratio $\frac{m_n(F)}{\diam^{h_n}(F)}$ does not change after passing from $F$ to $F'$, the estimate \eqref{eq:ratio_weaker} applies to $F$ as well. The proof of Lemma~\ref{prop: dowolny przedzial} is complete.
\end{proof}

Theorem~\ref{thm: estimate from above} now follows immediately from Lemma~\ref{prop: dowolny przedzial} and the formula \eqref{eq: gestosc liniowych}.

\section{Asymptotics of Hausdorff measure. Final conclusion.}
Here we formulate and prove the main result of Part 1 of our paper.

\begin{theorem}[Exact asymptotics of Hausdorff measure] \label{thm: wszystko razem}

Recall that $G$ is the linear analogue of  Gauss map. For every $n\in\N$ consider the iterating function system $G_n$ consisting of $n$ initial maps $g_1,\dots g_n$, and its limit set $J_n$,
i.e. the  set consisting of all points $x\in [0,1]$ such that the trajectory ${\{G^n(x)\}}_{n\in\mathbb N}$ never enters the interval $\left[0,\frac{1}{n+1}\right]$.

Recall that $h_n$ is the Hausdorff dimension of the limit set $J_n$ and $H_n$ is 

the Hausdorff measure of $J_n$ evaluated at its Hausdorff dimension, i.e. $H_n=H_{h_n}(J_n)$. Then

\[
\lim \limits_{n \to \infty } \frac{1-H_n}{1-h_n} \cdot \frac{1}{\ln n} = 1.
\]
Thus, 
\[
\lim \limits_{n \to \infty} \frac{n\cdot(1-H_n)}{\ln n}  = \frac{1}{\chi},
\]
where, we recall, $\chi$ is the Lyapunov exponent of the system $G$ with respect to the Lebesgue measure, i.e.
$$
\chi=\int_0^1\log|G'(x)|dx=\sum_{n=1}^\infty   \frac{1}{n(n+1)}\log (n(n+1)).
$$
\end{theorem}

\begin{proof}
Theorem~\ref{thm: estimate from below} and Theorem~\ref{thm: estimate from above} give the first formula.

The second formula is a straightforward conclusion from the first one and the fact that Hausdorff dimension $h_n$ has the asymptotics described in  \eqref{eq:asymp_dim for linear Gauss}. Indeed:
\[
\lim \limits_{n \to \infty} \frac{n(1-H_n)}{\ln n} 
= \lim \limits_{n \to \infty} \frac{n(1-H_n) \cdot (1-h_n)}{(1-h_n) \ln n} 
= \frac{1}{\chi}.
\]
\end{proof}

\part{Asymptotics of Hausdorff measure for the non-linear Gauss map: Lower bound}\label{part:2}


In this Part 2, we deal with estimates  of the Hausdorff measure for the original nonlinear Gauss map.
This means that now our iterated function system is formed by the continuous/real analytic inverse branches of the Gauss map 
map $F:(0,1]\to [0,1]$,  that is piecewise defined by the following 
infinite collection of maps:
$$
f_k:\left [\frac{1}{k+1},\frac{1}{k}\right ]\longrightarrow [0,1],
$$
where 
$$
f_k(x):=\frac{1}{x}-\left [\frac 1 x\right ]
$$
for every $k\in \N$.

The inverse maps $g_k=f_k^{-1}\colon [0,1]\rightarrow [0,1]$, $k\in\N$, forming the iterated function system $G$ are given by the formulas
$$
g_k(x)=\frac {1}{k+x}.
$$

As in previous sections, we consider the limit  set $J_n$ generated by $n$ initial maps of the collection $(g_k)_{k\in\mathbb N}$:
\[
J_n = \bigcap\limits_{k=1}^\infty \bigcup_{\omega\in \{1, 2,\ldots, n\}^k}g_{\omega}([0,1]).
\]

But now, the set $J_n$ has a straightforward and important interpretation in terms of continued fraction expansion: $J_n$  is exactly the set of such points $x\in [0,1]\setminus\mathbb Q$ for which the items in continued fraction expansion are bounded by $n$.

We keep notation from previous sections, i.e. we denote by $h_n$ Hausdorff dimension of the set $J_n$ and,
as in the linear case, 
we also put
 $$H_n:=H_{h_n}(J_n).$$  Let us recall that
\beq\label{220260226}
0<H_n<\infty,
\eeq
and as in previous sections, we denote by $m_n$ the normalized measure ${H_{h_n}}|_{J_n}$.

\section{ Hausdorff measure for finite truncations of the Gauss system: The upper bound by $1$}\label{sec:non_linear_est_above}
We now consider the original Gauss map $G$.  
We shall start with the following observation, analogous to Proposition~\ref{lem: miara mniejsza niz 1}.
The proof, however, is  not as straightforward as the proof of Proposition~\ref{lem: miara mniejsza niz 1}, because now the maps we are dealing with are not affine. We first need some distortion preparations.

Let $I$ be an interval in $\mathbb C$. We denote by $\mathbb D_I$ the disc with diameter $I$. We record the following easy observation.

\begin{lemma}\label{lem: conformal_extension}
There exists $\xi>0$ such that each transformation $g_k$, $k\in\N$, maps the disk
$$\mathbb D _{[-\xi,1+\xi]}$$ conformally into itself.
Therefore, each map $g_\omega$ maps 
    the disc 
$$\mathbb D _{[-\xi,1+\xi]}$$ conformally into itself.
\end{lemma}

Using the standard Koebe distortion theorem to the maps $g_\omega$ we can  estimate the distortion of $g_\omega$ on the interval $g_n^2(\Delta_n)$:

\begin{lemma}[Distortion estimate]\label{lem:distortion gomega}
There is a constant $C>0$ independent of $n$ such that for every map $g_\omega$ and all $z,w\in g_n^2(\Delta_n)$ we have that

$$
\frac{|g_\omega'(z)|}{|g_\omega'(w)|}\le \left (1+\frac{C}{n^4}\right ).
$$
 \end{lemma}
\begin{proof}
Choose an arbitrary point $z\in g_n^2(\Delta_n)$. Use the Koebe distortion estimate for the conformal map 
$$
g_\omega:\mathbb D_{[-\xi, 1+\xi]}\longrightarrow \mathbb D_{[-\xi,1+\xi]}
$$ 
and apply the fact that if $w$ is another point in $g_n^2(\Delta_n)$, then $|z-w|<\frac{1}{n^4}$. 
\end{proof}

The first core result in this section is the following.

\begin{prop}\label{prop:gauss_hausdorff_above}
\[H_{h_n}(J_n)\le 1.\]
\end{prop}
for all sufficiently large $n\in \mathbb N$.

\begin{proof}
For every $n\in\N$ denote:
\beq
\label{120260129}
D_n:=\left[\frac{1}{n+1},1\right].
\eeq 
Then $m_n(D_n)=1$, whence
$$\frac{\diam^{h_n}(D_n)}{m_n(D_n)}<1.$$

Applying the map $g_n$ to the interval $D_n$ we see that

$$\diam(g_n(D_n))=\frac{1}{\frac{1}{n+1}+n}-\frac{1}{n+1}=\frac{n+1}{1+n(n+1)}-\frac{1}{n+1}=\frac{n}{(n+1)(1+n+n^2)},$$
and
$$m_n(g_n(D_n)=\int_{D_n}\frac{1}{(x+n)^{2h_n}} d m_n(x)\ge\min_{x\in D_n} \left \{\frac{1}{(x+n)^{2h_n}}\right\}\ge \left (\frac{1}{1+n}\right )^{2h_n}.
$$
Thus,

\begin{equation}\label{eq:est_firstgn}
\bal
\frac{\diam^{h_n}(g_n(D_n))}{m_n(g_n(D_n))}
&\le \left (\frac{n(n+1)^2}{(n+1)(n^2+n+1)}\right )^{h_n}=\left (\frac{n^2+n}{n^2+n+1}\right )^{h_n}
\\
&=  \left (1-\frac{1}{n^2+n+1}\right )^{h_n}< 1.
\eal
\end{equation}
The distortion of the map $g_n$, i.e. the supremum of the ratio $\frac{|g_n'(y)|}{|g_n'(x)|}$, on the interval $\left[\frac{1}{n+1},\frac{1}{n}\right]=g_n([0,1])$ can be easily calculated. Indeed,
since $|(g_n)'(x)|=\frac{1}{(n+x)^2}$, we see that 
$$
\sup\left\{\frac {|g_n'(x)|}{|g_n'(y)|}:x,y\in\left[\frac{1}{n+1}, \frac{1}{n}\right]\right\}
\le \left (\frac{n+\frac 1 n}{n+\frac {1}{n+1}}\right )^2 =\left (1+\frac{1}{n^3+n^2+n}\right )^2.
$$
Applying the map $g_n$ to the interval $g_n(D_n)$,  and using the above estimate of distortion,  together with \eqref{eq:est_firstgn}, we obtain that

\begin{equation}\label{eq: secondgn}
\frac{\diam^{h_n}(g_n^2(D_n))}{m_n(g_n^2(D_n))}\le \left (1-\frac{1}{n^2+n+1}\right )^{h_n}\cdot \left (1+\frac{1}{n^3+n^2+n}\right )^{4h_n}<\left(1-\frac{1}{2n^2}\right)^{h_n}
\end{equation}
if $n$ is sufficiently large.

Next, we shall consider the compositions $g_\omega\circ g_n$ where $\omega=(\omega_1, \omega_2,\dots \omega_k)$ is an arbitrary finite word with entries bounded above by $n$. Combining the estimate \eqref{eq: secondgn} together with Lemma~\ref{lem:distortion gomega} we obtain that

\begin{equation}\label{eq:gauss_comp_above}
\frac{\diam^{h_n}\big(g_\omega \circ g_n^2(D_n)\big)}
{m_n(g_\omega\circ  g_n(D_n))}
\le \left (1-\frac{1}{2n^2}\right )^{h_n}
\cdot \left (1+\frac{C}{n^4}\right )^{2h_n}\le \left (1-\frac{1}{3n^2}\right )^{h_n}
\end{equation}
for all $n\in\N$ sufficiently large.

In order to continue the proof of Proposition~\ref{prop:gauss_hausdorff_above}, we shall use  again Corollary~\ref{cor:fundamental}.

It is known (see, e.g. \cite{MU}) that for every $n$ there exists a Borel  invariant probability measure $\mu_n$, equivalent to $m_n$, and ergodic.
Invariance means that for every Borel  set $A\subset [0,1]$ we have
$$\mu_n(A)=\sum_{j=1}^n\mu_n(g_j(A)),$$
while ergodicity says that for every Borel set $A\subset[0,1]$ such that $\bigcup_{j=1}^n g_j(A)=A$, we have either $\mu_n(A)=0$ or $\mu_n(A)=1$.

Let $\sigma:\mathbb N^{\mathbb N}\to\mathbb N^{\mathbb N}$ be the shift map, i.e. $\sigma(\omega)$ is
defined as a sequence such that  for every $n\in\mathbb N$ its $n$th
coordinate is equal to $\omega_{n+1}$. We denote by $\pi(\omega)$ the unique
element of $[0,1]$ whose continued fraction representation is equal to
$\omega$. So, we have defined an injective  map
$\pi:\mathbb N^{\mathbb N}\to[0,1]$. Its restriction to $\{1,\dots n\}^{\mathbb N}$ (equipped with product of discrete topologies on each $\{1,\dots n\}$) is then a homeomorphism  onto $J_n$.

Denote by $\tilde m_n$ the image of $m_n$
under the inverse of $\pi|_{\{1,\dots n\}^{\mathbb N}}$, and by $\tilde\mu_n$-- the image of $\mu_n$ under the inverse of $\pi|_{\{1,\dots n\}^{\mathbb N}}$.
Then $\tilde \mu_n$ is invariant and ergodic with respect to the shift map $\sigma$.
Now for every $\omega\in\{1,2,\dots ,n\}^{\mathbb N}$
let 
$$
Z_n(\omega):=\{j\ge 1:\omega_j=n \  \text{and}  \  \omega_{j+1}=n\}.
$$
Because of Birkhoff's Ergodic Theorem and  ergodicity of the measure
$\tilde\mu_n$, 
there exists a
Borel set $\Gamma_n\subset \{1,\dots n\}^{\mathbb N}$ of full measure $\tilde m_n$ (and full measure $\tilde\mu_n$)
such that for every $\omega\in\Gamma_n$ the set 
$Z_n(\omega)$
is infinite.

Pick such $\omega\in \Gamma_n$. Denote 
$$z:=\pi(\omega).$$
For each $j\in Z_n(\omega)$ we write
the initial segment of length $j$ of the infinite sequence $\omega$:
$$\omega_1, \omega_2,\dots, \omega_{j-1}, n, n.$$
Denote by $\omega|_{j-1}$ the finite word  
$$
\omega_1,\dots, \omega_{j-1}.
$$
Put $$F_j:=g_{\omega|_{j-1}}\circ g_n^2(D_n).$$

Then $F_j$ is a closed interval containing the point $z$. Using the estimate \eqref{eq:gauss_comp_above} we see that 

\begin{equation}\label{eq:est_density}
\frac{\diam^{h_n}(F_j)}
{m_n(F_j)}
\le \left (1-\frac{1}{3n^2}\right )^{h_n}.
\end{equation}
Since the set $Z_n(\omega)$ is infinite, one can find intervals $F_j$ containing the point $z$ with arbitrarily small diameters and satisfying the inequality \eqref{eq:est_density}. Together with density theorem for Hausdorff measure, i.e. Corollary~\ref{cor:fundamental}, this completes the proof of Proposition~\ref{prop:gauss_hausdorff_above}.

\end{proof}

\begin{rem}
The proof of Proposition~\ref{prop:gauss_hausdorff_above} gives slightly more than $H_n<1$; namely, the final estimate gives that

$$H_n\le\left (1-\frac{1}{3n^2}\right )^{h_n}.$$
\end{rem}
\section{Asymptotics of Hausdorff measure: Lower bound}\label{sec:non_linear_asymp_below}

For the purpose of this section, we introduce the following notation.
Given $\varepsilon\in (0,1)$ and $n\in \N$, we denote
$$
F_n(\varepsilon):=\left[b_{n+1},b_{[n-n^{1-\varepsilon}]+1}\right].
$$

\subsection{Estimates at large scale} 

We start with the following.

\begin{lemma}
\label{l320260129}
If $\varepsilon\in (0,1)$, then
\[
\varliminf\limits_{n \to \infty} \frac{1}{(1-h_n)\ln{n}}
\left(\frac{m_n(F_n(\varepsilon))}  {{\rm{diam}}^{h_n}(F_n(\varepsilon))}-1\right)
\ge 1-\varepsilon.
\]
\end{lemma}

\begin{proof}
Denote 
$$
n_{\varepsilon}:=n-n^{1-\varepsilon}.
$$
For every $[n_{\varepsilon}]+1\le j\le n$, we have 
$$
\bal
m_n({\Delta_j})
&=\int_0^1\frac{1}{(x+j)^{2h_n}}\, dm_n(x)
\ge (j+1)^{-2h_n}
=(j(j+1))^{-h_n}\left(\frac{j}{j+1}\right)^{h_n}
\\
&=\left(\frac{j}{j+1}\right)^{h_n}|{\Delta_j}|^{h_n}
=\left(1-\frac{1}{j+1}\right)^{h_n}|{\Delta_j}|^{h_n}
\\
&\ge \left(1-\frac{1}{n_{\varepsilon}}\right)^{h_n}|{\Delta_j}|^{h_n}.
\eal
$$
Hence, 
$$
m_n(F_n(\varepsilon))
\ge \left(1-\frac{1}{n_{\varepsilon}}\right)^{h_n}\sum_{j=[n_\varepsilon]+1}^n|{\Delta_j}|^{h_n}.
$$
Therefore, denoting $w_j=\frac{\Delta_j}{|F_n(\varepsilon)|}$,
$$
\bal
\frac{m_n(F_n(\varepsilon))}{|F_n(\varepsilon)|^{h_n}}
&\ge \left(1-\frac{1}{n_{\varepsilon}}\right)^{h_n}
\frac{\sum \limits_{j = [n_{\varepsilon}]+1}^{n} |{\Delta_j}|^{h_n}}{|F_n(\varepsilon)|^{h_n}}
\\
&= \left(1-\frac{1}{n_{\varepsilon}}\right)^{h_n}
\sum \limits_{j = [n_{\varepsilon}]+1}^{n}w_j^{h_n}.
\eal
$$
Hence,
$$
\bal
\frac{m_n(F_n(\varepsilon))}{|F_n(\varepsilon)|^{h_n}}-1
&\ge \left(1-\frac{1}{n_{\varepsilon}}\right)^{h_n}
\sum \limits_{j = [n_{\varepsilon}]+1}^{n}w_j^{h_n}-1
\\
&=\left(1-\frac{1}{n_{\varepsilon}}\right)^{h_n}
\left(\sum \limits_{j = [n_{\varepsilon}]+1}^{n}w_j^{h_n}-1\right) + \left(1-\frac{1}{n_{\varepsilon}}\right)^{h_n}-1.
\eal
$$
Therefore, noting \eqref{eq:asymp_dim for linear Gauss} and applying Lemma~\ref{l120260127} with $s_n=h_n$, we get 
\beq
\label{720260130}
\bal
\varliminf\limits_{n \to \infty}&\frac{1}{(1-h_n)\ln{n}}
\left(\frac{m_n(F_n(\varepsilon))}{|F_n(\varepsilon)|^{h_n}}-1\right) \ge
\\
&\ge \varliminf\limits_{n \to \infty}
\left(1-\frac{1}{n_{\varepsilon}}\right)^{h_n} \frac{1}{(1-h_n)\ln{n}}\left(\sum \limits_{j = [n_{\varepsilon}]+1}^{n}w_j^{h_n}-1\right)
+\varliminf\limits_{n \to \infty} \frac{\left(1-\frac{1}{n_{\varepsilon}}\right)^{h_n}-1}{(1-h_n)\ln{n}}
\\
&\ge 1-\varepsilon+\varliminf\limits_{n \to \infty} \frac{\left(1-\frac{1}{n_{\varepsilon}}\right)^{h_n}-1}{(1-h_n)\ln{n}}.
\eal
\eeq

Now, by the Mean Value Theorem,
$$
1-\left(1-\frac{1}{n_{\varepsilon}}\right)^{h_n}
\le \frac{h_n}{n_{\varepsilon}}\left(1-\frac{1}{n_{\varepsilon}}\right)^{h_n-1}
=h_n\frac{(n_{\varepsilon}-1)^{h_n-1}}{n_{\varepsilon}^{h_n}}
\le h_n(n_{\varepsilon}-1)^{-1}
\le \frac{1}{n_{\varepsilon}-1}.
$$
So, invoking \eqref{eq:asymp_dim for nonlinear Gauss}, we get that
$$
\varlimsup\limits_{n \to \infty}\frac{1-\left(1-\frac{1}{n_{\varepsilon}}\right)^{h_n}}{(1-h_n)\ln{n}}\le 0,
$$
or, equivalently,
$$
\varliminf\limits_{n \to \infty} \frac{\left(1-\frac{1}{n_{\varepsilon}}\right)^{h_n}-1}{(1-h_n)\ln{n}}\ge 0.
$$
Inserting this to \eqref{720260130} completes the proof of Lemma~\ref{l320260129}.
\end{proof}

\subsection{Any scale estimates}

\begin{lemma}[Distortion estimate]\label{lem:distortion gomega2}
There is a constant $C\in (0,+\infty)$ independent of $n\in\mathbb N\setminus\{1\}$ and $m\in\mathbb N$ such that for every map $g_\omega$, $\omega\in\N^*$, and all $z,w\in g_n^m([0,1])$, we have that
$$
\frac{|g_\omega'(z)|}{|g_\omega'(w)|}\le 1+\frac{C}{n^{2m}}.
$$
 \end{lemma}

\begin{proof}
The proof is the same as that of Lemma~\ref{lem:distortion gomega}. The only difference is that now $|z-w|<n^{-2m}$.
\end{proof}

We shall prove the following.

\begin{lemma}\label{prop:asymp_local_nonlinear}
Fix $\beta\in (0,1)$. Then for all $n$ sufficiently large the following holds: 

For $m_n$--a.e. $x\in J_n$ there exists an infinite  sequence of intervals $F_k(x)$, $k\in\mathbb N$, such that $x\in F_k(x)$ for all $k\in\mathbb N$, $\lim_{k\to\infty}\diam(F_k(x))=0$, and
$$
\frac{\frac{m_n(F_k(x))}{\diam^{h_n}(F_k(x))}-1}{(1-h_n)\ln n}\ge\beta.
$$
for all $k\in\mathbb N$.
\end{lemma}

\begin{proof}
We start by fixing $\varepsilon\in (0,1)$ such that $1-\varepsilon\in (\beta,1)$. Then we fix $\eta \in (\beta,1-\varepsilon)$. Then there exists $n_0\in\mathbb N$ so large that
$$\eta\cdot \left (1+\frac{C}{n^{4}}\right )^{-2}>\beta$$ for all $n\ge n_0$,
where $C$ is the constant coming from Lemma~\ref{lem:distortion gomega}. We abbreviate
$$
\hat n_\varepsilon:=[n-n^{1-\varepsilon}]+1=[n_\varepsilon]+1
$$

Recall that  every point $x\in J_n$ is uniquely determined by an infinite sequence $\omega=\omega(x)$,  $\omega\in \{1,\dots,n\}^\mathbb N$, i.e $x=\pi(\omega)$.  
Then for $m_n$-a.e. $x\in J_n$ we see, exactly as in the proof of Proposition~\ref{prop:gauss_hausdorff_above}, that  for infinitely many indices $j_k$ (depending on $x$)  the set 
$$F_k(x):=g_{\omega|_{j_k-1}}\circ g_n(F_n(\varepsilon))$$ 
contains $x$. 
Here $\omega=\omega(x)$ is the unique sequence which determines $x\in J_n$.

We shall bound from below  the ratio 
$$
\frac{m_n(F_k(x))}{\diam^{h_n}(F_k(x))}.
$$
First, notice that we have the following formulas:

$$
\begin{aligned} 
m_n\big(g_n(F_n(\varepsilon))\big)
&=\int_ { F_n(\varepsilon)}\frac{1}{(x+n)^{2h_n}}dm_n(x)
\ge\min_ {F_n(\varepsilon)}\left\{\frac{1}{(x+n)^{2h_n}}\right\} m_n(F_n(\varepsilon))\\
&\ge  \frac{1}{(n+1)^{2h_n}}\cdot  m_n(F_n(\varepsilon)),
\end{aligned}
$$
and
$$
\diam\big(g_n( F_n(\varepsilon))\big)
=\frac{1}{n+\frac{1}{n+1}}-\frac{1}{n+\frac{1}{\hat n_\varepsilon}}=\frac{n+1-\hat n_\varepsilon}{(n^2+n+1)(n\hat n_\varepsilon+1)},
$$
and
$$\diam (F_n(\varepsilon))=\frac{n+1-\hat n_\varepsilon}{(n+1)\hat n_\varepsilon}.$$
Therefore,
$$\frac{m_n(g_n(  F_n(\varepsilon)))}{\diam^{h_n} (g_n( F_n(\varepsilon)))}
\ge \left [   \frac{1}{(n+1)^2} \left (\frac{(n^2+n+1)(n\hat n_\varepsilon+1)}{(n+1)\hat n_\varepsilon} \right )\right]^{h_n}\cdot \frac{m_n( F_n(\varepsilon)) }{(\diam  F_n(\varepsilon))^{h_n}}.$$
The expression in the bracket is equal to
$$\frac{n^3\hat n_\varepsilon+n^2\hat n_\varepsilon+n^2+n\hat n_\varepsilon+n+1}{(n^3+3n^2+3n+1)\hat n_\varepsilon}=1-\frac{2n^2\hat n_\varepsilon+2n\hat n_\varepsilon+\hat n_\varepsilon-n^2-n-1}{(n^3+3n^2+3n+1)\hat n_\varepsilon}\ge 1-\frac 3 n $$
for all $n\in\N$ sufficiently large. Thus, for all such $n\in\N$ sufficiently large we have
$$
\begin{aligned}
\frac{\frac{m_n(g_n(  F_n(\varepsilon)))}{\diam^{h_n} (g_n( F_n(\varepsilon)))}-1}{(1-h_n)\ln n}
&\ge\frac{\frac{m_n( F_n(\varepsilon))}{\diam^{h_n} (g_n( F_n(\varepsilon)))}(1-\frac 3 n) -1}{(1-h_n)\ln n} \\
&=\left (1-\frac 3 n\right )\left (\frac {\frac{m_n(F_n(\varepsilon))}{\diam^{h_n} (g_n( F_n(\varepsilon)))}-1}{(1-h_n)\ln n} \right )-\frac 3 n\frac{1}{(1-h_n)\ln n}
\ge \eta.
\end{aligned}
$$
Therefore,
$$
\begin{aligned}
\frac{\frac{m_n(F_k(x))}{\diam^{h_n} (F_k(x))}-1}{(1-h_n)\ln n}
&\ge  \frac{\frac{m_n(g_n(  F_n(\varepsilon)))}{\diam ^{h_n}}(g_n( F_n(\varepsilon))\left (1-\frac {C}{n^4}\right )^{-2}-1}{(1-h_n)\ln n}\\
&=\frac{\frac{m_n(g_n(  F_n(\varepsilon)))}{\diam^{h_n}(g_n( F_n(\varepsilon))}-1}{(1-h_n)\ln n} \left (1-\frac {C}{n^4}\right )^{-2}+\frac{\left (1-\frac {C}{n^4}\right )^{-2}-1}{(1-h_n)\ln n}
\\
&>\eta \left (1-\frac {C}{n^4}\right )^{-2}+ \frac{\left (1-\frac {C}{n^4}\right )^{-2}-1}{(1-h_n)\ln n}
\ge \beta
\end{aligned}
$$
for all $n$ large enough.
\end{proof}

We are ready to prove the final theorem of this section.

\begin{theorem}\label{thm:asymp_gauss}
For the Gauss map we have that
$$\liminf_{n\to\infty} \frac{1-H_n}{(1-h_n)\ln n}\ge 1.$$
Thus,
$$\liminf_{n\to\infty} \frac{n(1-H_n)}{\ln n}\ge \frac{6}{\pi^2}.$$
\end{theorem}
\begin{proof}
Since, by Corollary \ref{cor:fundamental}, we have that
\[
\frac{1}{H_n}= \lim\limits_{r \to 0} \sup  \left \{\frac{m_n(F)}{{\rm{diam}}^{h_n}(F)} : x\in F, F\subset [0,1] -\text{interval}, \diam(F) < r\right \}.
\]
for $H_{h_n}$ a.e. point $x\in J_n$, it follows from Lemma~\ref{prop:asymp_local_nonlinear} that
$$\liminf_{n\to\infty} \frac{1-H_n}{(1-h_n)\ln n}\ge \beta$$
for every $\beta\in (0,1)$. So, letting $\beta$ grow to $1$, ends the proof of Theorem~\ref{thm:asymp_gauss}.
\end{proof}
\section{The upper bound result}
Recall that  $J_n$ be the limit set of the IFS generated by $n$ initial maps $g_k, k\le n$, Equivalently, $J_n$ is the set of those  irrational numbers in $[0,1]$ for which the continued fraction expansion has entries bounded by $n$.  
Denote $h_n=\dim_H(J_n)$. Hausdorff measure of the set $J_n$, evaluated at its Hausdorff dimension, i.e. 

$$H_n:=H_{h_n}(J_n)$$
is positive and finite.

In what follows  we prove the following (more difficult) estimate from above.

\begin{theorem}\label{thm:gauss_above}
$$
\limsup_{n\to\infty}\frac{1-H_n}{(1-h_n)\ln n}\le 1, 
$$
where, we recall, that for every $n\in\N$, $h_n$ is the Hausdorff dimension of the limit set $J_n$ of the iterated function system generated by the first $n$ initial analytic inverse branches of the (non-linear) Gauss map.
\end{theorem}

Along with Theorem~\ref{thm:asymp_gauss}, this gives the following main result of our paper.
\begin{theorem}\label{t120260312}
For the (non-linear) Gauss map we have
$$
\lim_{n\to\infty}\frac{1-H_n}{(1-h_n)\ln n}=1
$$
and
$$
\lim_{n\to\infty} \frac{n(1-H_n)}{\ln n}= \frac{6}{\pi^2}.
$$
\end{theorem}

\section{Transfer operators and their perturbations: Spectral properties}

Let $\BV$ be the vector space of all functions $f:(0,1)\to\mathbb C$ that have finite total variation $V_{(0,1)}(f)$. We consider its vector subspace $\BV_N$ consisting of all functions $f\in \BV$ such that
$$
f(x)=\frac 1 2 \big(f(x^+)+f(x^-)\big)
$$
for all $x\in (0,1)$, where $f(x^+)$, $f(x^-)$ respectively are the right-hand side and left-hand side limits of $f$ at $x$. It is well-known and not hard to see that $\BV_N$ becomes a Banach space if endowed with the norm
$$
\|f\|_{\BV}:=|f(0)^+|+V_{(0,1)}(f).
$$
It is immediate from this definition that
$$
\|f\|_{\infty}\le \|f\|_{\BV}.
$$
Fixing $t\in\mathbb C$ such that $ \re (t)>\frac 3 2$ and $n\in\N$, for every bounded function \\ $f:(0,1)\to \mathbb C$ define the bounded functions $\LL_{t,\infty}(f), \LL_{t,n}(f):(0,1)\longrightarrow \mathbb C$ by the respective formulas
\begin{equation}\label{eq:Lt,infty}
\LL_{t,\infty}(f)(x)
:=\sum_{k=1}^\infty f(g_k(x)) |g_k'(x)|^t
=\sum_{k=1}^\infty f\left (\frac{1}{x+k}\right )\cdot \frac{1}{(x+k)^{2t}}
\end{equation}
and
\begin{equation}\label{eq:Lt,n}
\LL_{t,n}(f)(x)
:=\sum_{k=1}^n f(g_k(x)) |g_k'(x)|^t
=\sum_{k=1}^n f\left (\frac{1}{x+k}\right )\cdot \frac{1}{(x+k)^{2t}}.
\end{equation}

For a bounded linear operator $L:\BV_N\to \BV_N$ we denote by $\|L\|_{\BV}$ its operator norm of $L$, i.e.,
$$
\|L\|_{\BV}=\sup\big\{\|L(f)\|_{\BV}: f\in\BV_N \  \text{and}  \  \|f\|_{\BV} \le 1\big\}.
$$
The following fact can be found in \cite{hensley}.

\begin{prop}[Lemma 1 in \cite{hensley}]
For every $t\in \mathbb C$ such that $\re( t)> \frac 3 4$ and every $n\in\mathbb N$ the operators $\mathcal L_{t,\infty}$ and $\mathcal L_{t,n}$ preserve the Banach space $BV_N$ and, moreover,
$$\mathcal L_{t,\infty}:\BV_N\to\BV_N  \  \ 
\text{and} \  \
\mathcal L_{t,n}:\BV_N\to\BV_N$$
are bounded linear operators.

Furthermore, for every $R\in (0,+\infty)$,
$$\max\Big\{\sup\big\{\|\mathcal L_{t,n}\|_{\BV}\big\},  \sup\{\|\mathcal L_{t,\infty}\|_{\BV}\big\}\Big\}<+\infty,
$$
where both suprema are taken over all $n\in\mathbb N$, all $t\in\mathbb C$ such that  $\re (t)>\frac 3 4$ and $|t|<R$.
\end{prop}

We now collect some significant spectral properties of the operators $\mathcal L_{t,\infty}$ and $\mathcal L_{t,n}$.

Recall that we denoted by $h_n$ the Hausdorff dimension of the limit set $J_n$. We are especially interested in the properties of the operators $\mathcal L_{h_n,n}$.

The following proposition is a fairly direct combination of Lemma~2 and Lemma 3  in \cite{hensley}, together with \eqref{eq:asymp_dim for nonlinear Gauss} which is, we recall, the main result in \cite{hensley} providing the asymptotics of the dimension $h_n$.

\begin{lemma}\label{prop: close_operators}
There exist $C_0\in (0,+\infty)$ and $n_0\in\mathbb N$ such that
$$\|\mathcal L_{h_n, n}-\mathcal L_{1,\infty}\|_{\BV}\le  C_0 (1-h_n).$$
for $n\ge n_0$.
\end{lemma}
\begin{proof}
Lemma 2 in \cite{hensley} tells that
\beq
\label{120260219}
\norm{\mathcal L_{t,\infty}-\mathcal L_{t,n}}_{\BV}\le 8|t|n^{1-2\cdot \re  t}
\eeq
for all $t\in \mathbb C$ such that $\re (t)>\frac 3 4$  and all $n\in\mathbb N$ while Lemma 3 in \cite{hensley} provides the estimate
\beq
\label{220260219}
\norm{\mathcal L_{s,n}-\mathcal L_{t,n}}_{\BV}\le  44|s-t| 
\eeq
for all $s, t\in \mathbb C$ such that $\re(s),\re(t) >\frac 3 4$, $|s|, |t| <\frac 3 2 $, and all $n\in\mathbb N\cup\{\infty\}$. Thus,
$$
\begin{aligned}
\norm{\mathcal L_{h_n,n}-\mathcal L_{1,\infty}}_{\BV}
&\le \norm{\mathcal L_{h_n, n}-L_{h_n,\infty}}_{\BV}+\norm{\mathcal L_{h_n,\infty}-\mathcal L_{1,\infty}}_{\BV}
\\
&\le 44(1-h_n)+8h_n\cdot n^{1-2h_n}
\\
&=44(1-h_n)+5h_n \cdot n^{-1}\cdot n^{2(1-h_n)}.
\end{aligned}
$$
Since also \eqref{eq:asymp_dim for nonlinear Gauss} implies that $\lim_{n\to\infty} n^{2(1-h_n)}=1$, Lemma~\ref{prop: close_operators} follows.
\end{proof}

Spectral properties of $\mathcal L_{1,\infty}$ are well known. See   \cite{hensley}, for precise description  of the action  of $\mathcal L_{1,\infty}$ on the specific spaces of holomorphic fuctions defined by Babenko and Mayer and the derivation of the following result, which we need in our work, see Lemma 6 in \cite{hensley}.

\begin{theorem}
\label{t220260219}
Number $1$ is an isolated element of the spectrum of the operator $\mathcal L_{1,\infty}: \BV_N\longrightarrow\BV_N$. It is a simple eigenvalue of $\mathcal L_{1,\infty}$ with an eigenfunction given by the formula
$$
g_{1,\infty}(x)=\frac {1}{\ln2}\frac{1}{x+1}.  
$$

The rest of the spectrum of the operator $\mathcal L_{1,\infty}: \BV_N\longrightarrow\BV_N$ is contained in the disk $\mathbb D(0,\eta)$ with some $\eta\in (0,1)$. 

In particular, the spectral radius of the operator $\mathcal L_{1,\infty}: \BV_N\longrightarrow\BV_N$ is equal to $1$.
\end{theorem}

This spectral portrait allows us to use a perturbation theorem, known in the literature as Kato-Rellich perturbation theorem. Hensley in his work  \cite{hensley} refers to the  
work \cite{cr} of Crandall and  Rabinowitz, which provides a related version. We refer also to the book \cite{OS} for the proof of the following version and additional comments.
Before formulating the Kato-Rellich Perturbation Theorem, we bring up the notion and notation of Riesz projections relevant in our context.
If $V$ is a complex Banach space, $L:V\to V$ is a bounded linear operator, and $\lambda_0\in\sigma(L)$ is an isolated point of the spectrum, then
the Riesz projection  corresponding to $\lambda$ is defined as
$$E_L(\lambda_0):=\frac{1}{2\pi i}\int_\gamma (\lambda I-L)^{-1}d\lambda.$$
where $\gamma$ is a positively oriented circle centred at $\lambda_0$, $\gamma=\partial^+(B(\lambda_0, \rho))$, with $\rho$  sufficiently  small, so that $\overline B(\lambda_0,\rho)\cap \sigma (L)=\{\lambda_0\}$. The name projection  is justified since $E_L(\lambda_0)\circ E_L(\lambda_0)=E_L(\lambda_0)$.

\begin{theorem}[Kato-Rellich Perturbation Theorem]
\label{t120260219}
Let $V$ be a complex  Banach space.
Let  $L_0:V\to V$ be a bounded linear operator for which $\lambda_0\in\mathbb C$ is an isolated point in the spectrum $\sigma(L)$ and it is a simple eigenvalue of $L$.

Then, there exists $\delta>0$ such that if $L:V\to V$ is a bounded linear operator $L:V\to V$ with $\|L-L_0\|<\delta$ then the intersection
$$
\sigma(L)\cap \mathbb D_{\mathbb C}(\lambda_0,\delta)
$$
is a singleton, whose only element which we denote by $\lambda_L$, is a simple isolated eigenvalue of $L$. 

Furthermore, after decreasing $\delta>0$ if necessary, the following statements hold. 

\ben
\item The function 
$$
B(L_0,\delta)\ni L\longmapsto \lambda_L\in \mathbb C
$$
is holomorphic.

\item The function 
$$
B(L_0,\delta)\ni L\longmapsto E_L(\lambda_L)
$$
taking values in the Banach space of all bounded linear operators from $V$ to $V$ is holomorphic.

\item The Riesz projection corresponding to the spectral set $\sigma(L)\setminus \{\lambda_L\}$ of $L$ is also a holomorphic function of the operator $L\in B(L_0,\delta)$.

\item For every vector  $v\in V\setminus (E_{L_0}(\lambda_{L_0}))^{-1}(0)$  and every $L$ with $L\in B(L_0,\delta)$  the vector 
$$
v(L):=E_L(\lambda_L)(v)
$$ 
is an eigenvector of $L:V\to V$ corresponding to the eigenvalue $\lambda_L$ and the function
$$
B(L_0,\delta)\ni L\longmapsto v(L)\in V
$$
is holomorphic. 
\een
\end{theorem}

\begin{rem}
In the above theorem the norm $\|L-L_0\|$ is  the operator norm in the space of bounded linear operators acting on $V$. The ball $B(L_0,\delta)$ is also taken with respect to this norm.    
\end{rem}

\begin{rem}
In the above perturbation theorem a  notion of a holomorphic function defined on a complex Banach space appears.  Let us recall one possible characterization of holomorphic maps in this setting.  Let two complex Banacha spaces $V$ an $W$ (finite or infinite- dimensional), equippped with norms $\|\cdot\|_V$ and $\|\cdot \|_W$ respectively, be given. Let $U\subset V$ be an open subset of $V$.  A function $f:U\to W$ is  holomorphic if 
and only if $f$ is continuous and $f|_{U\cap \tilde V}$ is holomorphic for every finite-dimensional linear subspace $\tilde V\subset V$. 
We refer the reader to the book \cite{mujica} for a full treatment and to the book \cite{OS} for a concise summary.
\end{rem}

Since the function $L\mapsto \lambda_L$ is  holomorphic in a neighbourhood of $\mathcal L_{1,\infty}$, as an immediate consequence of Theorem~\ref{t120260219} and \ref{t220260219}, along with \eqref{120260219} and \eqref{220260219}, we get the following.

\begin{cor}[perturbation of $\mathcal L_{1,\infty}$]\label{cor:perturb}
There exist $\delta\in (0,+\infty)$ and $C\in (0,\infty)$ such that if $t\in \mathbb C$ with $\re (t)>\frac 3 4$ is sufficiently close to $1$ and $n\in\N$ (including $\infty$) is sufficiently large, then the the following statements hold.

\ben
\item The operator $\mathcal L_{t,n}:\BV_N\to\BV_N$ has a simple isolated eigenvalue $\lambda_{t,n}$, satisfying 
$$
|\lambda_{t,n}-1|
\le C\|\mathcal L_{t,n}-\mathcal L_{1,\infty}\|_{\BV}
\le C\big(n^{-1}+|t-1|\big)
$$ 
and 
$$
\sigma\big(\mathcal L_{t,n}\big)\cap B(1,\delta)=\{\lambda_{t,n}\}.
$$

\item 
For every $g\in\BV_N$ the projection 
\begin{equation}\label{eq:projection}
E_{{\mathcal L}_{t,n}}(\lambda_{t,n})(g)
\end{equation}
is an eigenfunction of operator $\mathcal L_{t,n}:\BV_N\to\BV_N$ corresponding to $\lambda_{t,n}$.

\item
$$\norm{E_{\mathcal L_{t,n}}(\lambda_{t,n})-E_{\mathcal L_{1,\infty}}(\lambda_{1,\infty})}_{\BV}
\le C\|\mathcal L_{t,n}-\mathcal L_{1,\infty}\|_{\BV}
\le C\big(n^{-1}+|t-1|\big). 
$$
Consequently,
$$
\|E_{{\mathcal L}_{t,n}}(\lambda_{t,n})(g)-E_{{\mathcal L}_{1,\infty}}(\lambda_{1,\infty})(g)\|_{\BV}
\le C\big(n^{-1}+|t-1|\big)\|g\|_{\BV}.
$$
In particular, taking $g={\bf 1}:= {\bf 1}_{(0,1)}$, we get
$$
\|E_{{\mathcal L}_{t,n}}(\lambda_{t,n})({\bf 1})-E_{{\mathcal L}_{1,\infty}}(\lambda_{1,\infty})({\bf 1})\|_{\BV}
\le C\big(n^{-1}+|t-1|\big).
$$
\een
 \end{cor}

Closely related to Riesz projections is the notion of the conformal measures we have dealt "all the time" with in our paper. We are going to provide this relation now.

For every $t\in (3/4, +\infty)$ and $n\in\mathbb N$ consider the operator $\mathcal L_{t,n}$. This  operator corresponds to the finite Iterated Function System generated by the maps $g_1,\dots g_n$. It also acts continuously on the space of continuous functions $C([0,1])$.  The  conjugate (dual) operator 
$$
\mathcal L^*_{t,n}:C^*([0,1])\longrightarrow C^*([0,1])
$$
is defined by the formula
$$
\mathcal L^*_{t,n}(\nu)(g)=\nu\big(\mathcal L_{t,n}g\big).
$$
It is known (see \cite{CLUW, GDMS, MRUII, OS}) that for all $n\ge 2$ (including $\infty$) the spectral radius of this operator is its eigenvalue converging to $1$ as $n\to\infty$ and the corresponding eigenspace is $1$-dimensional vector space spanned over $\mathbb C$ by the real, positive, bounded, separated from zero real analytic function $\rho$.

Hence, 
$$
\rho|_{(0,1)}, \mathcal L_{t,n}\rho|_{(0,1)}\in \BV_N,
$$
and we conclude from Corollary~\ref{cor:perturb} that these spectral radii are for all $t>3/4$ sufficiently close to $1$ and all $n\in\mathbb N$ sufficiently large equal to the eigenvalues $\lambda_{t,n}$ and the functions $\rho|_{(0,1)}$ are complex multiples of $E_{{\mathcal L}_{t,n}}(\lambda_{t,n})({\bf 1})$. It also follows from the above sources that for all such $t$ and $n$ the number $\lambda_{t,n}$ is an eigenvalue of the conjugate operator $\mathcal L^*_{t,n}:C^*([0,1])\longrightarrow C^*([0,1])$, its corresponding eigenspace is $1$-dimensional, and
there exists a unique Borel probability eigenmeasure for this eigenvalue. We denote it by $m_{t,n}$. In a formula:
$$
\mathcal L_{t,n}^*(m_{t,n})=\lambda_{t,n}m_{t,n}.
$$
Therefore there exists a unique eigenfunction $\rho_{t,n}\in \BV_N$  of $\mathcal L_{t,n}:\BV_N\to\BV_N$ corresponding to the eigenvalue $\lambda_{t,n}$ such that 
\beq
\label{120260224}
m_{t,n}(\rho_{t,n})=1.
\eeq
In addition, the function $\rho_{t,n}$ is a real, uniformly bounded, positive, uniformly separated from zero real-analytic function. The measure $m_{t,n}$ is called $t$-conformal for the iterated function system generated by the maps $g_1,\dots,g_n$. In particular, $m_{h_n,n}$ is just the familiar conformal measure we have been dealing with and it is the normalized $h_n$-dimensional Hausdorff measure on $J_n$.

We shall prove the following.

\begin{prop}[relation between Riesz projection and the functional $m_{t,n}$]\label{prop:projection_conf_measure}
If $t\in (3/4, +\infty)$ is sufficiently close to $1$ and $n\in\N$ is sufficiently large (including $\infty$), then 
$$
E_{\mathcal L_{t,n}}(\lambda_{t,n})(g)=m_{t,n}(g)\cdot \rho_{t,n}
$$
for every $g\in\BV_N$. 

In particular,  
$$
\rho_{t,n}=E_{\mathcal L_{t,n}}(\lambda_{t,n})({\bf 1})
$$ 
and
$$
\|\rho_{t,n}-\rho_{1,\infty}\|_\BV
\le C\big(n^{-1}+|t-1|\big),
$$
where $C$ is the constant coming from Corollary ~\ref{cor:perturb}.
\end{prop}

\begin{proof}
Abbreviate
$$
E_{t,n}:=E_{\mathcal L_{t,n}}(\lambda_{t,n}).
$$ 
Because of Corollary ~\ref{cor:perturb} there exists a function $\ell:\BV_N\to \mathbb C$ such that
$$
E_{t,n}(g)=\ell(g)E_{t,n}({\bf 1})
$$
for every $g\in\BV_N$. Obviously, $\ell$ is a linear functional and
$$
\ell({\bf 1})=1.
$$
Also,
$$
\ell\big(\mathcal L_{t,n}g\big)E_{t,n}({\bf 1})
=E_{t,n}(\mathcal L_{t,n}g)
=\mathcal L_{t,n}E_{t,n}(g)
=\lambda_{t,n}E_{t,n}(g)
=\lambda_{t,n}\ell(g)E_{t,n}({\bf 1}).
$$
Thus,
$$
\mathcal \ell(L_{t,n}g)=\lambda_{t,n}\ell(g).
$$
So, $\ell$ is a complex multiple of $m_{t,n}$, i.e. there exists a complex number $c$ such that
$$
\ell=cm_{t,n}.
$$
Hence,
$$
1=\ell({\bf 1})=cm_{t,n}({\bf 1})=c.
$$
Thus, 
$$
\ell=m_{t,n},
$$
yielding
$$
E_{t,n}(g)=m_{t,n}(g)E_{t,n}({\bf 1})
$$
for every $g\in\BV_N$. So,
$$
\rho_{t,n}
=E_{t,n}(\rho_{t,n})
=m_{t,n}(\rho_{t,n})E_{t,n}({\bf 1})
=E_{t,n}({\bf 1}).
$$
Therefore, the first two assertion of our propositions are proved. Having it the third (last) assertion directly follows from item (3) in Corollary~\ref{cor:perturb}. We are done.
\end{proof}

\begin{prop}\label{prop:conformal_measures_close}
If $t\in (3/4, +\infty)$ is sufficiently close to $1$ and $n\in\N$ is sufficiently large (including $\infty$), then
$$
\big|m_{1,\infty}(g)- m_{t,n}(g)\big|
\le C\left(2+4\ln 2 {\|E_{1,\infty}\|_{\BV}}\right)\big(n^{-1}+|t-1|\big)\|g\|_{\BV}
$$
for every $g\in \BV_N$. 

Treating, if needed, $m_{1,\infty}$ and $m_{t,n}$ as elements of the conjugate Banach space $\BV_N^*$ consisting of  bounded linear functionals defined on the Banach space $\BV_N$, this inequality can also be written as
$$
\big\|m_{1,\infty}- m_{t,n}\big\|_{\BV_N^*}
\le C\left(2+ 4\ln 2{\|E_{1,\infty}\|_{\BV}}\right)\big(n^{-1}+|t-1|\big).
$$
\end{prop}

\begin{proof}
We again use the abbreviation
$$
E_{t,n}:=E_{\mathcal L_{t,n}}(\lambda_{t,n}).
$$ 
Let $g\in  BV_N$. According to  Proposition~\ref{prop:projection_conf_measure}, we have
\begin{equation}
\label{eq:measure_difference}
\begin{aligned}
m_{1,\infty}(g)&- m_{t,n}(g)
=E_{1,\infty}({g})\cdot\frac{1}{\rho_{1,\infty}}-E_{t,n}({g})\cdot \frac{1}{\rho_{t,n}}=
\\
&=E_{1,\infty}({g})\cdot\frac{1}{\rho_{1,\infty}}-E_{1,\infty}({g})\cdot \frac{1}{\rho_{t,n}}
+E_{1,\infty} ({g})\cdot\frac{1}{\rho_{t,n}}-E_{t,n}({g})\cdot \frac{1}{\rho_{t,n}}
\\
&=E_{1,\infty}({g})\left(\frac{1}{\rho_{1,\infty}}- \frac{1}{\rho_{t,n}}\right)
+\big(E_{1,\infty} ({g})-E_{t,n}({g})\big)\frac{1}{\rho_{t,n}}.
\end{aligned}
\end{equation}

Now, it follows from \eqref{120260224} that there exists $x_{t,n}\in (0,1)$ such that
\beq
\label{120260226}
\rho_{t,n}(x_{t,n})\ge 1/2.
\eeq
Then, using also Proposition~\ref{prop:projection_conf_measure} and the fact that
$$
\rho_{1,\infty}(x)=\frac {1}{\ln 2}\frac{1}{x+1}\ge \frac{1}{2\ln 2},
$$
we get 
$$
\bal
\left|\frac{1}{\rho_{1,\infty}(x_{t,n})}- \frac{1}{\rho_{t,n}(x_{t,n})}\right|
&=\left|\frac{\rho_{t,n}(x_{t,n})-\rho_{1,\infty}(x_{t,n})}{\rho_{t,n}(x_{t,n}) \rho_{1,\infty}(x_{t,n})}\right|
\\
&\le 4 \ln 2\left|\rho_{t,n}(x_{t,n})-\rho_{1,\infty}(x_{t,n})\right|
\le 4 \ln 2\|\rho_{t,n}-\rho_{1,\infty}\|_\BV
\\
&\le 4 C\cdot {\ln 2}\big(n^{-1}+|t-1|\big).
\eal
$$
Hence,
\beq
\label{220260224}
\bal
\Bigg|E_{1,\infty}({g})(x_{t,n})&\left(\frac{1}{\rho_{1,\infty}}- \frac{1}{\rho_{t,n}}\right)(x_{t,n})\Bigg|=
\\
&=\big|E_{1,\infty}({g})(x_{t,n})\big|
\left|\frac{1}{\rho_{1,\infty}(x_{t,n})}- \frac{1}{\rho_{t,n}(x_{t,n})}\right|
\\
&\le 4 {C}{\ln 2}\|E_{1,\infty}({g})\|_{\BV}\big(n^{-1}+|t-1|\big)
\\
&\le  4 {C}{\ln 2}\|E_{1,\infty}\|_{\BV}\big(n^{-1}+|t-1|\big)\|g\|_{\BV}.
\eal
\eeq

Using item (3) in Corollary~\ref{cor:perturb} and \eqref{120260226}, we get
$$
\bal
\left|\big(E_{1,\infty} ({g})-E_{t,n}({g})\big)(x_{t,n})\frac{1}{\rho_{t,n}(x_{t,n})}\right|
&=\Big|\big(E_{1,\infty} ({g})-E_{t,n}({g})\big)(x_{t,n})\Big|\left|\frac{1}{\rho_{t,n}(x_{t,n})}\right|
\\
&\le 2\Big\|\big(E_{1,\infty} ({g})-E_{t,n}({g})\big)\Big\|_{\BV}
\\
&\le 2C\big(n^{-1}+|t-1|\big)\|g\|_{\BV}.
\eal
$$
Inserting this and \eqref{220260224} to \eqref{eq:measure_difference}, we get
$$
\big|m_{1,\infty}(g)- m_{t,n}(g)\big|
\le \left(2+ 4 \ln 2{\|E_{1,\infty}\|_{\BV}}\right)C\big(n^{-1}+|t-1|\big)\|g\|_{\BV}.
$$
The proof of Proposition~\ref{prop:conformal_measures_close} is complete.
\end{proof}

\section{Some Auxiliary Estimates in the Banach space $\BV_N$.}
We consider finite and infinite sequences
$\omega\in\mathbb N^k$, $\omega\in \mathbb N^{\mathbb N}$.
We also use the notation $\mathbb N^*$ to denote the collection of all finite sequences of positive integers, i.e.,
$$\mathbb N^* =\bigcup_{k\in\mathbb N}\mathbb N^k$$
We recall that for each such a finite sequence 
$\omega=(\omega_1,\dots \omega_k)$, we 
denoted
\begin{equation}\label{eq:g_omega}
 g_\omega=g_{\omega_1}\circ\dots \circ g_{\omega_k}.
\end{equation}

\begin{lemma}\label{lem:dist1}
There exists a constant $c\in (0,+\infty)$ such that
 \begin{equation}\label{eq:dist1}
|g'_\omega(z)-g'_\omega(w)|\le c|g_\omega'(z)|\cdot |z-w|
\end{equation}
for every $\omega\in \mathbb N^*$ and all $z,w\in [0,1]$.

In particular
\beq
\label{120260310}
\frac{|g'_\omega(z)|}{|g'_\omega(w)|}
\le \frac{\sup(|g'_\omega|)}{\inf(|g'_\omega|)}
\le 1+c.
\eeq
for every $\omega\in \mathbb N^*$ and all $z,w\in [0,1]$.
\end{lemma}

\begin{proof}
Each map $g_k(x)=\frac{1}{x+k}$, $k\in \N$ treated as a function of a complex variable $x$, is a M\"obius transformation with pole at $-k$. Let us recall that (see Lemma~\ref{lem: conformal_extension})  there exists $\xi\in (0,+\infty)$ such that 
$$
g_k\left(\mathbb D\left(\frac 1 2, \frac 1 2+\xi\right)\right)\subseteq \mathbb D\left(\frac 1 2, \frac 1 2+\xi\right)
$$
and the function
$$
g_k:\mathbb D\left(\frac 1 2, \frac 1 2+\xi\right)\longrightarrow \mathbb D\left(\frac 1 2, \frac 1 2+\xi\right)
$$
obviously is holomorphic and univalent. So, 
$$
g_\omega\left(\mathbb D\left(\frac 1 2, \frac 1 2+\xi\right)\right)\subseteq \mathbb D\left(\frac 1 2, \frac 1 2+\xi\right)
$$
and the function
$$
g_\omega:\mathbb D\left(\frac 1 2, \frac 1 2+\xi\right)\longrightarrow \mathbb D\left(\frac 1 2, \frac 1 2+\xi\right)
$$
is holomorphic and univalent for every $\omega\in \mathbb N^*$.

Then the classical Koebe's distortion theorem (see e.g. \cite{pom}) guarantees  that, in particular for all $w,z\in [0,1]$,
\begin{equation}\label{eq:distortion_est}
\left |\frac{g'_\omega(w)}{g'_\omega(z)}-1\right |\le c |z-w|
\end{equation}
where the constant $c$ depends only on the size of the ''margin'' $\xi$. This ends the proof of Lemma~\ref{lem:dist1}.
\end{proof}

\begin{lemma}
\label{prop:gomega_bounds}
There exists a constant $D\in (0,+\infty)$ such that
\begin{equation}
\| |g_\omega'|^h\|_{\BV}\le D \inf(|g'_\omega|^h)
\end{equation}
for every $h\in(1/2,1]$ and for every $\omega\in\mathbb N^*$.
\end{lemma}

\begin{proof}
Of course, it is enough to bound the total variation of the function $|(g_\omega)'|^h$. It follows from Lemma~\ref{lem:dist1} that
\begin{equation}\label{eq:distortion}
\frac{|g'_\omega(z)|}{|g'_\omega(w)|}\le 1+c|z-w|
\end{equation}
and, consequently, 
$$
\frac{|g'_\omega(z)|^h}{|g'_\omega(w)|^h}\le (1+c|z-w|)^h<1+c|z-w|$$
where the last inequality is due to the fact that $h<1$. Thus, using \eqref{eq:distortion} again, we get
$$
\begin{aligned}
\big | |g'_\omega(z)|^h-|g'_\omega(w)|^h\big |
&\le c|z-w|\cdot |g'_\omega(z)|^h \\
&\le c(1+c)\inf(|g'_\omega|^h)|z-w|.
\end{aligned}
$$
Therefore, the function $[0,1]\ni u \longmapsto |g_\omega'(u)|^h$ is Lipschitz continuous with the constant $\inf(|g'_\omega|^h)$. Hence, the total variation norm is bounded  by $D\inf(|g'_\omega|^h)$ with the constant $D:=c(c+1)$, thus independent of $\omega\in\mathbb N^*$.
\end{proof}
\section{Upper Estimates for Asymptotics of  Hausdorff measure}
Using \eqref{220260226} and \eqref{320260226}, it follows from the formula \eqref{eq:limitformula2B} in Corollary~\ref{cor:ratio} that 
\beq
\label{220260312}
\limsup_{n\to\infty} \frac{1-H_n}{(1-h_n)\ln n}
=\limsup_{n\to\infty}
\left (\lim_{r\to 0}\frac{\sup\left\{\frac { m_{h_n,n}(F)}{\diam^{h_n}(F)}\right\}-1}{(1-h_n)\ln n}\right ),
\eeq
where supremum is taken over all closed intervals $F\subseteq [0,1]$ containing $x$ with diameter less than or equal to $r$, and $m_{h_n,n}$ is, as usually in this part of the paper, the $h_n$-conformal measure for the Iterated Function System generated by contractions $g_1,\dots, g_n$; it coincides with the normalized Hausdorff measure $H^{1}_{h_n}$. 

We shall estimate the above supremum in several steps.

\subsection{Step 1: Upper estimates for the sets $g_\omega([0,1])$ and $g_\omega([b_{l+1}, b_k])$}

\begin{lemma}
\label{prop:passing1}
There exists a constant $C_1\in (0,+\infty)$ such that
$$
\frac{m_{h_n,n}\big(g_\omega([0,1])\big)}{\big|g_\omega([0,1])\big|^{h_n}}-1\le C_1n^{-1}
$$
for all $n\in \N$ and every $\omega\in \mathbb N^*$.
\end{lemma}

\begin{proof}
We have 
$$
\begin{aligned} 
m_{h_n, n}({g_\omega([0,1])})&-|{g_\omega([0,1])}|^{h_n}=
\\
&=m_{h_n, n}({g_\omega([0,1])})-m_{1,\infty}^{h_n}({g_\omega([0,1])})= 
\\
&=\int_0^1 |(g'_{\omega})|^{h_n}dm_{h_n,n}-\left(\int_0^1|g'_{\omega}|dm_{1,\infty}\right )^{h_n}
\\
&=\int_0^1 |(g'_{\omega})|^{h_n}dm_{h_n,n}
-\int_0^1|(g'_{\omega})|^{h_n}dm_{1,\infty} +
\\
&\hspace{3cm}+\int_0^1|(g'_{\omega})|^{h_n}dm_{1,\infty}
-\left(\int_0^1|g'_{\omega}|dm_{1,\infty}\right )^{h_n}.
\end{aligned}
$$
Since the function $[0,1]\ni x\longmapsto x^{h_n}\in [0,1]$ is concave, it follows from the Jensen's inequality that the second difference in the above formula is $\le 0$, yielding the inequality
$$ 
m_{h_n, n}({g_\omega([0,1])})-|g_\omega([0,1])|^{h_n}
\le \int_0^1 |(g'_{\omega})|^{h_n}dm_{h_n,n}
-\int_0^1|(g'_{\omega})|^{h_n}dm_{1,\infty}.
$$
Applying now Proposition~\ref{prop:conformal_measures_close} and Lemma~\ref{prop:gomega_bounds}, we get that
$$ 
\bal
m_{h_n, n}({g_\omega([0,1])})-|{g_\omega([0,1])}|^{h_n}
&\le C_2\inf\big(|g'_{\omega}|^{h_n}\big)\big(n^{-1}+(1-h_n)\big)
\\
&\le C_1n^{-1}|{g_\omega([0,1])}|^{h_n}
\eal
$$
with some constants $C_2, C_1\in (0,+\infty)$ independent of $n$ and $\omega$.
The proof of Lemma~\ref{prop:passing1} is complete.
\end{proof}

Since $g_\omega(\Delta_j)=g_{\omega j}([0,1])$, as an immediate consequence of this lemma we get the following.

\begin{cor}
\label{cor:gomegaIj}
There holds, with the constant $C_1$ coming from Lemma~\ref{prop:passing1}
\begin{equation}\label{eq:Ij}
\frac{m_{h_n,n}\big(g_\omega(\Delta_j)\big)}{\big|g_\omega(\Delta_j)\big|^{h_n}}-1\le C_1n^{-1}
\end{equation}
for all $n\in \N$, all $j\in\mathbb N$,  and every $\omega\in \mathbb N^*$.
\end{cor}

\begin{lemma} 
\label{prop:copies}
We have 
\begin{equation}
\frac{m_{h_n,n}\big(g_\omega([b_{l+1},b_k])\big)}{\big|g_\omega([b_{l+1},b_k])\big|^{h_n}}-1
\le (1-h_n)\ln n+O(n^{-1}).
\end{equation}
for all integers $1\le k\le l\le n$ and every $\omega\in \mathbb N^*$.
\end{lemma}

\begin{proof}
Since
$$
[b_{l+1},b_k]=\bigcup_{j=k}^l {\Delta_j},
$$
summing up the estimates \eqref{eq:Ij} in Corollary~\ref{cor:gomegaIj} over all $j=k,\dots,l$, looking up at \eqref{eq:asymp_dim for nonlinear Gauss}, \eqref{420260127}, and applying Lemma~\ref{l120260303} with
$$
u_j:=\frac{\big|g_\omega(\Delta_j)\big|}{\big|g_\omega([b_{l+1},b_k])\big|},
$$ 
we obtain
\begin{equation}\label{eq:summing}
\bal
m_{h_n,n}&\big(g_\omega([b_{l+1},b_k])\big)\le 
\\
&\le (1+C_1n^{-1})\sum_{j=k}^l\big|g_\omega(\Delta_j)\big|^{h_n} 
\\
&= (1+C_1n^{-1})\frac{\sum_{j=k}^l \big|g_\omega(\Delta_j)\big|^{h_n}}{\left(\sum_{j=k}^l \big|g_\omega(\Delta_j)\big|\right)^{h_n}} |g_\omega([b_{l+1},b_k])|^{h_n}
\\
&=(1+C_1n^{-1})|g_\omega([b_{l+1},b_k])|^{h_n}\sum_{j=k}^l u_j^{h_n}
\\
&=(1+C_1n^{-1})\Big(1+(1-h_n)\ln n+O\big(((1-h_n)\ln n)^2\Big )|g_\omega([b_{l+1},b_k])|^{h_n}
\\
&=\big(1+(1-h_n)\ln n+O(n^{-1})\big) |g_\omega([b_{l+1},b_k])|^{h_n}.
\eal 
\end{equation}
The proof of Lemma~\ref{prop:copies} is complete.
\end{proof}

Noting that Lemma~\ref{prop:copies} also holds for $l>n$, as its immediate consequence we get the following.

\begin{cor}\label{cor:0bk}
We have 
\begin{equation}
\frac{m_{h_n,n}\big(g_\omega([0,b_k])\big)}{\big|g_\omega([0,b_k])\big|^{h_n}}-1
\le (1-h_n)\ln n+O(n^{-1}).
\end{equation}
for all integers $k, n\ge 1$ and every $\omega\in \mathbb N^*$.
\end{cor}

\subsection{Step 2: Upper estimates for the intervals $g_\omega([0,r))$}

It is well-known and easy to see that there exists $\gamma\in (0,1)$ such that 
\begin{equation}
\label{320260303}
\big|g_\omega([0,1])\big|\le \gamma^{|\omega|}
\end{equation}
for all $\omega\in \mathbb N^*$.

In order to obtain upper estimates for an arbitrary interval $[0,r)$, we use the decomposition of $[0,r)$ obtained in the same way as in Lemma~\ref{lem: rozdzial przedzialu 0,r} but with the Gauss map instead of of its linear version. We thus represent the interval $[0,r)$ as a (infinite in general) union of intervals 
$$[0,r)=\bigcup_{i=1}^\infty I_i,$$
where each set $I_j$ is either degenerate (i.e. a single point)  or  it is an interval  of the  form $g_\omega([0, b_{k_j}])$ or $I_j=g_\omega([b_{k_j},1]$ and  the length of the word $\omega$, satisfies  $|\omega|=j-1$. The left endpoint of $I_j$ coincides with the right endpoint of $I_{j-1}$.
Recall also that in the construction at the $N$-th step we have 
$$
[0,r)\subset   I_1\cup I_2\cup\dots\cup I_N\cup I_{N+1}',
$$
where $I_{N+1}'$ is a cylinder sets of $N$-th generation adjacent to $I_N$.einyrt.lafar
\
\begin{lemma}
\label{prop:g_omega_r}
The following uniform estimate holds for every $r\in (0,1)$ and every $\omega\in \mathbb N^*$:
\begin{equation}\label{eq:0r}
m_{h_n}\big(g_\omega([0,r))\big) \le \left (1+(1-h_n)\ln n+O(n^{-1}\ln \ln n)\right ) |g_\omega([0,r))|^{h_n}.
\end{equation}
More precisely, there exists a constant $C_3\in (0,+\infty)$ such that for every $r\in (0,1)$, every $\omega\in \mathbb N^*$, and every $n\in\mathbb N$, we have 
$$ 
m_{h_n}\big(g_\omega([0,r))\big)\le \left (1+(1-h_n)\ln n+C_3n^{-1}\ln\ln n\right )\cdot |g_\omega([0,r))|^{h_n}. 
$$
\end{lemma}

\begin{proof}
We write 
$[0,r)\subset I_1\cup I_2\cup\dots \cup I_N\cup I_{N+1}'$, with some $N\le n$. The exact choice of $N$   will be determined later on.

It follows from Proposition~\ref{prop:copies} that
$$
{m_{h_n,n}\big(g_\omega(I_j)\big)} 
\le \big(1+(1-h_n)\ln n+O(n^{-1})\big)\cdot |g_\omega(I_j)|^{h_n}
$$
for every $j=1,\ldots, N$ and
$$
{m_{h_n,n}\big(g_\omega(I_{N+1}')\big)} \le \big(1+(1-h_n)\ln n+O(n^{-1})\big) |g_\omega(I_{N+1}')|^{h_n}
$$
Therefore, using also Lemma~\ref{l120260303} and \eqref{eq:asymp_dim for nonlinear Gauss}, we get
\begin{equation}
\label{eq:summingI}
\begin{aligned}
m_{h_n,n}&\Big(g_\omega\big(I_1\cup\dots\cup I_N\cup I'_{N+1}\big)\Big)\le
\\
&\le \big(1+(1-h_n)\ln n+O(n^{-1})\big)\sum_{j=1}^N |g_\omega(I_j)|^{h_n}+|g_\omega(I_{N+1}')|^{h_n}
\\
&=\big (1+(1-h_n)\ln n +O(n^{-1})\big)
\left (\sum_{j=1}^N |g_\omega(I_j)|+|g_\omega(I_{N+1}')|\right )^{h_n} \cdot 
\\
&\hspace{5.5cm}\cdot \frac{\sum_{j=1}^N |g_\omega(I_j)|^{h_n}+|g_\omega(I_{N+1}')|^{h_n}}{\left(\sum_{j=1}^N |g_\omega(I_j)|+|g_\omega(I_{N+1}')|\right)^{h_n}}
\\
&\le \left (1+(1-h_n)\ln n +O(n^{-1})\right )\cdot
\\
&\hspace{2cm}\cdot \Big(1+(1-h_n)\ln (N+1)+O\big((n^{-1}\ln (N+1))^2\big)\Big)\cdot
\\
&\hspace{6cm}\cdot \left (\sum_{j=1}^N |g_\omega(I_j)|+|g_\omega(I_{N+1}')|\right )^{h_n}
\end{aligned}
\end{equation}
To proceed further, we  need to estimate from above the product
\begin{equation}\label{eq:ratio}
\bal
|g_\omega([0,r))|^{-h_n}&\left(\sum_{j=1}^N |g_\omega(I_j)|+|g_\omega(I_{N+1}')|\right)^{h_n}=
\\
&=\left (|g_\omega([0,r))|^{-1}\left( \sum_{j=1}^N |g_\omega(I_j)|+|g_\omega(I_{N+1}')|\right )\right )^{h_n}
.
\eal
\end{equation}
We may assume that $r>\frac{1}{n+1}$ since otherwise the interval $[0,r)$ does not intersect the set $J_n$ and the estimate required in our lemma is obvious.
Since $\bigcup_{j=1}^N I_j\subset [0,r)$
and $|I'_{N+1}|$ is a cylinder set  of $N$-th  generation, thus, by \eqref{320260303}, of length not exceeding $\gamma^N$, and $r>\frac{1}{n+1}$, by virtue of \eqref{120260310} in Lemma~\ref{lem:dist1}, we get 
$$   
\bal
|g_\omega([0,r))|^{-1}&\left( \sum_{j=1}^N |g_\omega(I_j)|+|g_\omega(I_{N+1}')|\right )=
\\
&= \frac{\sum_{j=1}^N |g_\omega(I_j)|}{|g_\omega([0,r))|}
+\frac{|g_\omega(I_{N+1}')|}{|g_\omega([0,r))|}
\\
&\le 1+\frac{\sup(|g'_\omega|)}{\inf(|g'_\omega|)}\frac{|I_{N+1}'|}{|[0,r)|}
\le 1+(1+c){\gamma^N}(n+1).
\eal
$$
Taking $N:=\left[\frac{2}{-\ln\gamma}\ln (n+1)\right]+1$, we thus get 
$$   
|g_\omega([0,r))|^{-1}\left( \sum_{j=1}^N |g_\omega(I_j)|+|g_\omega(I_{N+1}')|\right )
\le 1+(1+c)n^{-1}
$$
and
$$(1+(1-h_n)\ln (N+1)+O(n^{-1}\ln (N+1))^2\le 1+ O((1-h_n)\ln\ln n).$$
Therefore, looking up at \eqref{eq:summingI} \eqref{eq:ratio}, we get
$$
\begin{aligned}
&m_{h_n,n}\big(g_\omega([0,r))\big)\le
\\
&\hspace{5mm}\le m_{h_n,n}\Big(g_\omega\big(I_1\cup\dots\cup I_N\cup I'_{N+1}\big)\Big)
\\
&\hspace{5mm}\le  \left (1+(1-h_n)\ln n+O(n^{-1}\right )\cdot \left (1+O((1-h_n)\ln\ln n)\cdot(1+(1+c)n^{-1}\right ) \cdot |g_\omega([0,r)]^{h_n}\\
&\hspace{5mm}\le \left (1+(1-h_n)\ln n+O(n^{-1}\ln\ln n)\right )|g_\omega([0,r))|^{h_n}.
\end{aligned}
$$
The proof of Lemma~\ref{prop:g_omega_r} is complete.
\end{proof}

\subsection{Step 3: Upper estimates for  the intervals $g_\omega((r,1])$}
    
\begin{lemma}\label{cor:copies-20260312}
The following uniform estimate holds for every $r\in (0,1)$ and every $\omega\in \mathbb N^*$:
\begin{equation}\label{eq:0r-2}
m_{h_n}\big(g_\omega((r,1])\big) \le \big (1+(1-h_n)\ln n+O(n^{-1}\ln \ln n)\big ) |g_\omega((r,1])|^{h_n}.
\end{equation}
More precisely, there exists a constant $C_4\in (0,+\infty)$ such that for every $r\in (0,1)$, every $\omega\in \mathbb N^*$, and every $n\in\mathbb N$, we have 
$$ 
m_{h_n}\big(g_\omega((r,1])\big)\le \left (1+(1-h_n)\ln n+C_4n^{-1}\ln\ln n\right )\cdot |g_\omega((r,1])|^{h_n}. 
$$
\end{lemma}

\begin{proof}
Given an interval $(r,1]$, let $k=k(r)\ge 1$ be the integer for which  $r\in (b_{k+1}, b_k]$. Then 
$$(r,1]=[b_k,1]\cup (r,b_k].$$
For the set $[b_k,1]$ we apply Lemma~\ref{prop:copies}, getting
$$
m_{h_n,n}\big(g_\omega([b_{k},1]\big)
\le(1+ \ln n (1-h_n)+O(n^{-1}))\cdot  {|g_\omega([b_{k},1])|^{h_n}}.
$$
Since $(r,b_k]=g_k([0,\hat r))$ with some $\hat r\in [0,1]$, using  Proposition~\ref{prop:g_omega_r}, we get
$$
m_{h_n,n}\left (g_\omega((r,b_k])\right ) \le \left (1+(1-h_n)\ln n+O(n^{-1}\ln \ln n)\right ) |g_\omega((r,b_k])|^{h_n}.
$$
Thus,
$$
\begin{aligned}
m_{h_n,n}&\big(g_\omega((r,1])\big)=
\\
&=m_{h_n,n}\big(g_\omega((r,b_k])\big)+m_{h_n,n}\big(g_\omega([b_{k},1])\big)
\\
&\le \left (1+(1-h_n)\ln n+O(n^{-1}\ln \ln n)\right )\cdot \big(|g_\omega((r,b_k])|^{h_n}+|g_\omega([b_{k},1])|^{h_n}\big)\\
&\le \left (1+(1-h_n)\ln n+O(n^{-1}\ln \ln n)\right )\cdot |g_\omega((r,1])|^{h_n}\cdot 2^{1-h_n}\\
&= \left (1+(1-h_n)\ln n+O(n^{-1}\ln \ln n)\right )\cdot |g_\omega((r,1])|^{h_n}.
\end{aligned}
$$
The proof of Lemma~\ref{cor:copies-20260312} is complete.
\end{proof}


\subsection{Upper estimates for an arbitrary interval}

In this section we obtain the estimates for an arbitrary interval contained in $[0,1]$.

\begin{prop}\label{prop: dowolny przedzialB}
If $F\subset [0,1]$ is an arbitrary closed interval, then 
$$
m_{h_n}(F) \le \left (1+(1-h_n)\ln n+O(n^{-1}\ln \ln n)\right ) |F|^{h_n},
$$
where the constant witnessing to $O(n^{-1}\ln\ln n)$ is independent of $F$ and $n$.

Thus, 
$$
\limsup \limits_{n\to\infty}
\frac{\sup\left\{\frac { m_{h_n,n}(F)}{\diam^{h_n}(F)}\right\}-1}{(1-h_n)\ln n}\le 1,
$$
where supremum is taken over all closed intervals $F\subseteq [0,1]$.
\end{prop}

\begin{proof}
Let $F \subset [0,1]$ be an arbitrary closed interval. Fix $n \in \N$. We can assume that $F\subset [b_{n+1}, 1]$ since otherwise setting 
$$
\widetilde{F}:= F \cap [b_{n+1},1],
$$
we have $m_n(\widetilde{F})=m_n(F)$, while ${\rm{diam}}\widetilde{F} \leq {\rm{diam}} F$. Consequently,
\[
\frac{m_n(\widetilde{F})}{{\rm{diam}}^{h_n}(\widetilde{F})} \geq \frac{m_n(F)}{{\rm{diam}}^{h_n}(F)}.
\]
Now consider two cases:

\vspace{1mm}
Case 1. $F$ contains some interval of the form $[b_{j+1},b_{j}]$, $j \leq n $.
 
Case~2. $F$ does not contain any interval of the form $[b_{j+1},b_{j}]$, $j \leq n $.
      
\vspace{2mm}
First, we will focus on the Case~1. Let $[b_l, b_k]$, $k < l \leq n+1$, be the union of all  intervals of the form $[b_{j+1},b_{j}]$, $j \leq n $, that are contained in $F$. Then $F$ can be represented as a union of three intervals
\[
F = I_1 \cup [b_l, b_k] \cup I_2
\]
where $I_1\subset [b_{l+1}, b_l]$ and $I_2\subset [b_k, b_{k-1}]$.

For each of these three intervals we proved the upper estimates in previous steps: the estimate for $[b_{l+1},b_k]$ is provided in Lemma ~\ref{prop:copies}, while the interval $I_1$ is of the form $g_{l}([0,r])$ and $I_2$ is of the form $g_k([1,s])$.  

Therefore, applying Lemma~\ref{prop:g_omega_r}, Lemma ~\ref{prop:copies}, and Lemma~\ref{cor:copies-20260312}, we obtain 
$$
\begin{aligned}
m_{h_n,n}(F)
&=m_{h_n,n}(I_1\cup [b_{l+1},b_k]\cup I_2)
\\
&\le 
\left (1+(1-h_n)\ln n+O(n^{-1}\ln \ln n)\right ) (|I_1|^{h_n}+|[b_{l+1},b_k]|^{h_n}+|I_2|^{h_n}\\
&\le  \left (1+(1-h_n)\ln n+O(n^{-1}\ln \ln n)\right ) |F|^{h_n}3^{1-h_n}\\
&\le  \left (1+(1-h_n)\ln n+O(n^{-1}\ln \ln n)\right ) |F|^{h_n}.
\end{aligned}
$$
We are done with Case~1.

Now, we move on to Case~2. Then, again there are two subcases:

\vspace{1mm}
Case~2a. $F \subsetneqq [b_{k+1}, b_{k}]$ for some $k \leq n$. 

Case~2b. $F = I_1 \cup I_2 $ where $I_1 \subsetneqq [b_k, b_{k-1}]$ and $I_2 \subsetneqq [b_{k-1}, b_{k-2}]$ for some $k \leq n$.

\vspace{2mm}
In Case~2b the same way of estimate applies as in Case~1.  Indeed, then
\beq\label{120260312}
F=g_{l}([0,r))\cup g_k([1,s])
\eeq
with some $k\le n$, $l\le n+1$, and  $r, s\in [0,1]$. So, 
the only difference to Case~1 is that now there are two summands instead of three, which leads to the (even better) factor $2^{1-h_n}$ instead of $3^{1-h_n}$.

In Case~2a, let $\omega\in\mathbb N^*$ be the longest word such that
$$
F \subsetneqq g_\omega([0,1]).
$$
We notice that being in this case, we have $|\omega|\ge 1$. Put 
$$
F':=g_{\omega}^{-1}(F).
$$
We conclude that either $F'$ falls into Case~1 or into Case~2b. 

Consider first the situation when $F'$ falls into Case 1. Then   
$$
F'=I_1'\cup [b_{l+1}, b_k]\cup I_2'
$$ 
with some $1\le k<l\le n+1$, where $I_1'\subset [b_{l+1},b_{l}]$ and $I_2'\subset [b_{k}, b_{k-1}] $. Hence, 
$$
F=g_\omega(I'_1)\cup g_\omega([b_{l+1}, b_k])\cup g_\omega(I_2').
$$
Note that $I_1'=g_{l+1}([0,r))$ and $I_2'=g_{k-1}((\tilde r, 1]$ with some $0<r,\tilde r<1$. So, 
$$
F=g_{\omega(l+1)}([0,r))\cup g_\omega([b_{l+1},b_k]\cup g_{\omega(k-1)} ((\tilde r,1]).
$$
Thus, applying now Lemma~\ref{prop:g_omega_r}, Lemma ~\ref{prop:copies}, and Lemma~\ref{cor:copies-20260312} respectively to the the first, second, and the third of these intervals, we obtain
$$
\begin{aligned}
m_{h_n,n}(F)
&\le  \left (1+(1-h_n)\ln n+O(n^{-1}\ln \ln n)\right )  \cdot (|I_1|^{h_n}+[b_{l+1},b_k]^{h_n}+|I_2|^{h_n})\\
&\le \left (1+(1-h_n)\ln n+O(n^{-1}\ln \ln n)\right )\cdot 3^{1-h_n}\cdot |F|^{h_n}\\
&=  \left (1+(1-h_n)\ln n+O(n^{-1}\ln \ln n)\right )\cdot |F|^{h_n}.
\end{aligned}
$$
Finally, we deal with the situation when $F'$ falls into Case~2b. Then, because of \eqref{120260312}, we have that
$$
F'=g_{l}([0,r'))\cup g_k([1,s'])
$$
with some $k\le n$, $l\le n+1$, and  $r', s'\in [0,1]$. Hence 
$$
F=g_\omega(F')=g_{\omega l}([0,r'))\cup g_{\omega k}([1,s'])
$$
and we conclude this case by exactly the same arguments as in original Case~2b.
\end{proof}

Having this proposition, the proof of Theorem~\ref{thm:gauss_above} directly follows from \eqref{220260312}.

\bibliographystyle{plain}

\end{document}